\documentclass[reqno]{amsart}
\usepackage[utf8]{inputenc}
\usepackage{subcaption}
\usepackage{amsmath}
\usepackage{amssymb}
\usepackage{amsfonts}
\usepackage{amsthm}
\usepackage{mathrsfs}
\usepackage[all]{xy}
\usepackage{graphicx}
\usepackage{color}
\usepackage{cite} 
\usepackage{url}
\usepackage{indent first}
\usepackage[labelfont=bf,labelsep=period,justification=raggedright]{caption}
\usepackage[english]{babel}
\usepackage[utf8]{inputenc}
\usepackage{hyperref}
\usepackage[colorinlistoftodos]{todonotes}
\usepackage{tkz-fct}
\usepackage{tikz}
\usetikzlibrary{calc}
\usepackage[enableskew]{youngtab}
\usepackage{pstricks}
\usepackage{multicol}
\usepackage{graphicx} 
\usepackage{fancyhdr}
\theoremstyle{plain}
\newtheorem{thm}{Theorem}
\newtheorem{lemma}[thm]{Lemma}
\newtheorem{cor}[thm]{Corollary}

\newtheorem{prop}[thm]{Proposition}

\theoremstyle{definition}
\newtheorem{defn}[thm]{Definition}

\theoremstyle{remark}
\newtheorem{rem}[thm]{Remark}
\newtheorem*{ex}{Example}

\numberwithin{equation}{section}
\numberwithin{thm}{section}

\newcommand{\CC}{\mathbb{C}}

\newcommand{\NN}{\mathbb{N}}

\newcommand{\RR}{\mathbb{R}}

\newcommand{\mcB}{\mathcal{B}}

\newcommand{\mcL}{\mathcal{L}}
\newcommand{\mcM}{\mathcal{M}}

\newcommand{\mcP}{\mathcal{P}}

\newcommand{\mcS}{\mathcal{S}}

\newcommand{\FP}{\text{FP}_\infty}

\title{A Recurrence Based Summation Method for Ultra-Rapid Divergent Series and Renormalon Type Expansions}
\author{Ishan Joshi}

\begin{document}

\begin{abstract}
Classical summation methods are often organized around particular growth regimes. Standard Borel summation is suited to Gevrey-1 series, while higher-order Gevrey behavior is commonly handled by changing the kernel, for instance through Mittag-Leffler summation. In this paper, we introduce a recurrence-based summation method, called C-summation, whose primary input is a first-order inhomogeneous recurrence.

The recurrence does not determine a unique solution, since different solutions may differ by a homogeneous term. We remove this ambiguity by passing to a normalized tail, where the homogeneous ambiguity becomes an additive constant, and then extracting the finite part of that tail. The resulting finite-part selector is defined through a Bromwich transform on the normalized tail differences. 

We prove that, under a precise $\mathcal M$-admissibility hypothesis, the resulting value is independent of the chosen solution of the recurrence and of the chosen admissible recurrence presentation of the same formal series. We also show that C-summation is regular and homogeneous, and that it is stable under an explicit shift-compatibility condition on the normalized tail. In a certain Borel-summable class, we prove agreement with Borel-Laplace summation.

\end{abstract}

\maketitle

\section{Introduction} \label{introduction}

Classical approaches to divergent series, most notably Borel-Laplace  \cite{costin2008asymptotics} summation and Écalle’s resurgent analysis \cite{ecalle1981fonctions}, are built on continuous integral transforms. While powerful, these methods are often tied to particular growth scales. For instance, standard Borel summation treats Gevrey-1 growth ($\sum n! z^n$) but fails for superfactorial divergence. To handle higher-order Gevrey-$k$ growth, one must often abandon the Borel kernel for a Mittag-Leffler kernel $E_\alpha(z)$, which requires tuning the index $\alpha$ to match the specific order of divergence. This dependence on the chosen kernel becomes especially visible for ultra-rapid divergence, where one often has to use additional tools such as acceleration operators. This suggests asking whether one can formulate a summation procedure from a different starting point, based on discrete structure rather than on a chosen transform kernel.

In this paper, we introduce $C$-summation, a summation method based on first-order inhomogeneous recurrences. We begin with a recurrence $ a(n)I_n+b(n)I_{n+1}=\mu_n$. Unfolding this recurrence produces an associated formal series. The recurrence does not determine a unique solution, as solutions can differ by homogeneous terms. To remove this ambiguity, we pass to a normalized tail. At that level, the ambiguity becomes additive. We then define the sum by extracting the finite part of an admissible asymptotic expansion of that tail.

This leads to a summation method with the following properties. We prove that $C$-summation is independent of the chosen solution of the recurrence and independent of the particular recurrence presentation of the same formal series. We also prove that it is regular and homogeneous. Stability under finite truncation is treated more carefully: it holds when the constant term extracted from the normalized tail does not change after shifting the tail forward by one step, but this is not automatic from existence alone. Thus the theory distinguishes between existence of the $C$-sum and the stronger requirement that deleting or reindexing finitely many initial terms behaves in the expected way.

A further point is that $C$-summation can be compared with Borel Summation. We prove that, under a specific collection of hypotheses on the Borel transform and the associated remainder terms, the $C$-sum agrees with the Borel-Laplace sum. This does not amount to a general equivalence with Borel summation, but it shows that the recurrence construction is compatible with it in the class covered by the theorem.

A point worth emphasizing is that the same construction, without modification, applies both to Gevrey-$k$ series for arbitrary $k$ and to the ultra-rapid series $\sum (-1)^n e^{n^2/2}$. In the latter case, the construction yields the value $1/2$.

Finally, we apply this framework to the problem of renormalons (discussed in \cite{beneke1999renormalons}), which are series which appear frequently in Quantum Chromodynamics \cite{t1979can}, and whose Borel transforms possess singularities on the positive real axis (e.g., $\sum n! a^n$). For the simple example treated here, the construction yields the principal value associated with the Borel transform without contour deformation.

\section{Preliminaries}

A summation method assigns values to a class of formal or divergent series. In discussing such methods, several structural properties are standard. We record the ones relevant for the present paper. Let $A: X \to Y$ be a summation method. 
\begin{defn} [Homogeneity] \label{homogeneity}
     $A$ is homogeneous if $A(k\sum a_n)= k\cdot A(\sum a_n)$
\end{defn} 

\begin{defn} [Linearity]
    $A$ is linear if $A(\sum a_n) + A(\sum b_n) = A(\sum (a_n +b_n))$
\end{defn}

\begin{defn} [Regularity] \label{regularity}
    $A$ is said to be regular if whenever $\sum a_n$ is convergent, $A(\sum a_n) = \sum a_n $. In other words, $A$ sums every convergent series to its ordinary sum. 
\end{defn}

\begin{defn} [Stability] \label{stability}
    If $A(\sum_{n=0}^\infty a_n) = S$, then $a_0 + A(\sum_{n=1}^\infty a_n) = S$.
\end{defn}

The main comparison in this paper is with Borel summation, so we recall the form that will be used later.
\begin{defn}[Borel Summation]
    Let $\sum_{n=0}^\infty a_n z^n$ be a formal power series, and let $\widehat{A}(t) = \sum_{n=0}^\infty \frac{a_n}{n!} t^n$ be its Borel transform. If $\widehat{A}$ extends analytically to a neighborhood of $[0,\infty)$, then the Borel sum of the series defined by $$\int_0^\infty e^{-t} \widehat{A}(zt) \ dt$$ when this integral converges. 
\end{defn}
This is the classical analytic procedure against which the recurrence-based method will be compared in the later sections.

\section{Main Results}

We now state the main theorem. It gives the basic criterion under which a recurrence presentation produces a well-defined value and records the summation-theoretic properties proved later in the paper.

\begin{thm}
Let $\mathcal{S}=\sum_{k=0}^\infty c_k$ be a formal series. Suppose $\mathcal{S}$ admits a first-order recurrence presentation $a(n)I_n+b(n)I_{n+1}=\mu_n$ with associated formal coefficients
$$c_k = (-1)^k \mu_{n_0+k}\prod_{j=0}^{k-1}\frac{b(n_0+j)}{a(n_0+j+1)},$$
and suppose that for some solution $(I_n)$ the normalized tail
$$
R_N = a(n_0)\frac{I_{n_0+N}}{P(n_0+N,n_0)}
\qquad\text{where}\qquad
P(n,n_0)=\prod_{k=n_0}^{n-1}\left(-\frac{a(k)}{b(k)}\right)
$$
is $\mathcal{M}$-admissible. Then the quantity
$$C(\mathcal{S}) := a(n_0)I_{n_0}-\mathrm{FP}_\infty(R)$$ exists and is well defined in the following sense:

\begin{enumerate}
    \item It is independent of the choice of solution $(I_n)$.
    \item It is independent of the chosen admissible recurrence presentation of $\mathcal{S}$.
    \item If $\mathcal{S}$ is convergent and $C$-summable, then $C(\mathcal{S})$ agrees with the ordinary sum.
    \item $C$-summation is homogeneous.
    \item If, in addition, the normalized tail satisfies the shift-compatibility condition $ \mathrm{FP}_\infty(R)=\mathrm{FP}_\infty(R^+)$, then $C$-summation is stable.
\end{enumerate}
\end{thm}

The rest of the paper is organized as follows. We begin with the continued-fraction viewpoint associated with the recurrence and then introduce the asymptotic framework needed to define admissibility and finite-part extraction. These ingredients are then used to prove the theorem above. After that, we compare the method with Borel-Laplace summation in a concrete class and conclude with examples, including Gevrey-type, ultra-rapid, and renormalon-type series.

\section{Continued Fractions} \label{continued fractions representing certain integrals}

Our discussion begins with the simple two term recurrence $$a(n)I_n + b(n)I_{n+1} = \mu_n$$ where $a(n),b(n): \NN \to \CC$. A natural way to study this recurrence is to pass the sequence $I_n$ to a ratio $r_n = I_{n+1}/I_n$. This turns the recurrence into an iteration by fractional linear transformations. In particular, it produces a generalized continued fraction, so the recurrence may be viewed not only as a linear relation among the $I_n$, but also as a nonlinear dynamic system. Continuing in this fashion, we have that $$\frac{a(n)I_{n} + b(n)I_{n+1}}{a(n+1)I_{n+1} + b(n+1)I_{n+2}} = \frac{\mu_n}{\mu_{n+1}}$$
Then, if we define $r_n = \frac{I_{n+1}}{I_n}$, dividing the numerator and denominator by $I_{n+1}$ results in
$$\frac{\frac{a(n)}{r_n} + b(n)}{a(n+1) + b(n+1)r_{n+1}} = \frac{\mu_n}{\mu_{n+1}}$$
Finally, separating the $r_n$ term, we get that: 
\begin{equation}
    r_n = \frac{a(n)\mu_{n+1}}{a(n+1)\mu_n - b(n)\mu_{n+1} + b(n+1)\mu_n r_{n+1}} \label{eq:1}
\end{equation}
This reformulation is useful for two reasons. First, it shows that continued fractions arise naturally from the recurrence rather than being introduced artificially. Second, it suggests that divergent behavior in the original series may still be organized by a stable recursive law at the level of the ratios. In favorable cases, the continued fraction may converge even when the associated series does not, and this raises the possibility that the recurrence carries meaningful information beyond ordinary convergence of the series itself.

The next theorem makes this connection explicit. By repeatedly unfolding the recurrence, one obtains a formal Eulerian expansion for a normalized initial value. This expansion is the series naturally attached to the recurrence, and it is a scaled version of this series that will later be assigned a value by $C$-summation.

\begin{thm} \label{formal series}
    If $\mu_{n_0} \neq 0$, then $\frac{a(n_0)I_{n_0}}{\mu_{n_0}}$ has the formal series expansion $$1+\sum_{k=1}^\infty (-1)^k\frac{\mu_{{n_0}+k}}{\mu_{n_0}}  \prod_{j=0}^{k-1} \frac{b({n_0}+j)}{a({n_0}+j+1)}$$  
\end{thm}

\begin{proof}
     Based on the recurrence, we know that $a({n_0})I_{n_0} + b({n_0})I_{{n_0}+1} = \mu_{n_0}$ . Then, $a({n_0})I_{n_0} = \mu_{n_0} - b({n_0})I_{{n_0}+1}$. Thus, $$\frac{a({n_0})I_{n_0}}{\mu_{n_0}} = 1- \frac{b({n_0})}{\mu_{n_0}}I_{{n_0}+1}$$ Now, let $T'_n=\frac{a(n)I_n}{\mu_n}$ Then, applying this to $I_{{n_0}+1}$, we get $I_{{n_0}+1} = \frac{\mu_{{n_0}+1}}{a({n_0}+1)}T'_{{n_0}+1}$. Then, we have that $$T'_{n_0} = 1- \frac{b({n_0})}{\mu_{n_0}}I_{{n_0}+1} = 1- \frac{b({n_0})\mu_{{n_0}+1}}{a({n_0}+1)\mu_{n_0}}T'_{{n_0}+1}$$ Then, iterating this identity, we get the following formal series:
    \begin{align}T'_{n_0} = \frac{a({n_0})I_{n_0}}{\mu_{n_0}} =& \frac{1}{1+\frac{b({n_0})}{a({n_0})}r_{n_0}} = 1+\sum_{k=1}^\infty \prod_{j=0}^{k-1} -\frac{b({n_0}+j)\mu_{{n_0}+j+1}}{a({n_0}+j+1)\mu_{{n_0}+j}}\\ =& \ 1+\sum_{k=1}^\infty (-1)^k\frac{\mu_{{n_0}+k}}{\mu_{n_0}}  \prod_{j=0}^{k-1} \frac{b({n_0}+j)}{a({n_0}+j+1)} \label{eq:2} \end{align} 

    This is a formal expansion obtained through a recursive substitutions. Thus, no convergence claim is made here.

\end{proof}

$\eqref{eq:2}$ is central to the subsequent analysis of divergent series, as sums of this form will appear frequently during our discussion of divergent series.

\section{The Recurrence Ambiguity} \label{the reccurence ambiguity}
The previous section sets up a recurrence of the form $a(n)I_n + b(n)I_{n+1}=\mu_n$. However, there is a problem. There are multiple solutions to this recurrence, all of which differ by a homogeneous term. To see this, suppose $J_n$ is another solution to the recurrence. Let $H_n = J_n - I_n$. We therefore have that $a(n)H_n + b(n)H_{n+1} =0$, meaning that $$H_{n+1} = -\frac{a(n)}{b(n)}H_n = -\frac{a(n)}{b(n)}\cdot -\frac{a(n-1)}{b(n-1)} H_{n-1} = H_{n_0}\prod_{k=n_0}^{n-1}- \frac{a(k)}{b(k)} $$ Then, $$J_n = I_n+(J_{n_0} - I_{n_0})\prod_{k=n_0}^{n-1} -\frac{a(k)}{b(k)}$$ For simplicity, denote $P(n,n_0)=\prod_{k=n_0}^{n-1} -\frac{a(k)}{b(k)}$.  This computation shows that there are infinite solutions to this recurrence, all of which differ by some homogeneous term. From this a natural question arises. How do you choose the "correct" solution of the recurrence? What does it mean for a solution to even be "correct"? In order for the summation method to be well defined, this question must be addressed. 

The ambiguity $J_n=I_n+C\,P(n,n_0)$ shows that the recurrence alone does not determine $I_{n_0}$. 

\section{The Finite Part}\label{finite term}

The previous section shows that the recurrence
$$
a(n)I_n+b(n)I_{n+1}=\mu_n
$$
does not determine a unique solution. Throughout this section, whenever a recurrence presentation is fixed, we assume $a(n)\neq 0,$ and $b(n)\neq 0$ for all $n\ge n_0$. Any two solutions differ by a homogeneous term. Thus, for a summation method coming from the recurrence, there must be a principled way to resolve this ambiguity. Define the normalized tail sequence to be
$$
R_N=a(n_0)\frac{I_{n_0+N}}{P(n_0+N,n_0)}
$$
where
$$
P(n,n_0)=\prod_{k=n_0}^{n-1}\left(-\frac{a(k)}{b(k)}\right).
$$

This normalized tail is the right object to study the ambiguity because the homogeneous term becomes a constant shift at the level of $R_N$. Prior to normalization, different solutions differ by a large oscillatory homogeneous factor. After normalization, they differ by a constant. Thus, it becomes natural to find a procedure to extract the constant contribution from the normalized tail.

In general, however, the sequence $R_N$ need not converge. Even in simple examples, it may contain logarithmic growth, inverse powers of $N$, exponential scales, alternating phases such as $(-1)^N$, or mixtures of several such behaviors. Ordinary limits are therefore too restrictive: the normalized tail may fail to have a limit, while still possessing enough structure to allow an intrinsic finite-part extraction.

The construction in this section has three steps. First, we define a class $\mcM$ of coefficient germs which is large enough to contain the tail differences appearing in the examples, but rigid enough that an eventual sequence has at most one representation in the class. Second, for a coefficient germ $c\in\mcM$, we construct a canonical selector $U_c$ satisfying
$$
U_c(N)-U_c(N+1)=c(N)
$$
for all sufficiently large $N$. Third, we define the finite part of a normalized tail $R_N$ by subtracting this selector from $R_N$. The result is eventually constant, and that eventual constant is the finite part.

We first define the coefficient-side class used in the present paper. The class below is an exact finite-sum class. There are no hidden remainders. This is what makes subtraction closure and the later uniqueness argument rigorous.

\subsection{Coefficient germs}

We begin with the positive log-exp scales.
\begin{defn}[Reduced log-exp monomial scale]
Set
$$
\ell_0(x):=x,\qquad \ell_{j+1}(x):=\log \ell_j(x)
$$
for all sufficiently large $x$. Let $\mcP$ denote the power-log monomial group
consisting of germs
$$
P(x)=\prod_{j=0}^r \ell_j(x)^{\alpha_j},
\qquad
\alpha_j\in\RR,
$$
with coefficient equal to $1$. We define the reduced log-exp monomial scale $\mcL$ recursively. First set $$\mcL_0:=\mcP.$$ Suppose $\mcL_{\le h}:=\bigcup_{q=0}^h\mcL_q$ has been defined. A reduced
exponent of height at most $h$ is a finite sum
$$
\Phi(x)=\sum_{\nu=1}^s c_\nu M_\nu(x),
$$
where $c_\nu\in\RR^\times$, the $M_\nu\in\mcL_{\le h}$ are distinct, each
$M_\nu(x)\to+\infty$, and no $M_\nu$ is one of the iterated logarithms
$\ell_j$ for $j\ge 1$. The empty sum is allowed. Then define $\mcL_{h+1}$ to be
the set of all germs of the form
$$
L(x)=\exp(\Phi(x))P(x),
\qquad
P\in\mcP,
$$
where $\Phi$ is a reduced exponent of height at most $h$. Finally set
$$
\mcL:=\bigcup_{h\ge 0}\mcL_h.
$$

For $L_1,L_2\in\mcL$, write
$$
L_1\prec L_2
\qquad\Longleftrightarrow\qquad
\frac{L_1(x)}{L_2(x)}\to 0
\qquad (x\to+\infty).
$$
\end{defn}
Elements of $\mcL$ are reduced log-exp expressions, viewed as eventual germs on the positive real axis. The reduced expression is part of the data. Thus $\mcL$ consists only of expressions of the prescribed recursive form. Positive constants are excluded by requiring the power-log factor $P$ to have coefficient $1$. Lower-order exponential unit terms are excluded because the exponent
$\Phi$ only contains monomials $M_\nu$ satisfying $M_\nu(x)\to+\infty$. Iterated-logarithm terms are not allowed in $\Phi$, since they are represented instead through the power-log factor $P$.

The next lemma says that this scale is ordered by growth. This is the basic comparison result used throughout the finite-part construction.

\begin{lemma}[Trichotomy for reduced log-exp monomials]
For $L_1,L_2\in\mcL$, exactly one of
$$
L_1\prec L_2,\qquad L_2\prec L_1,\qquad L_1=L_2
$$
holds, where equality means equality as eventual germs.
\end{lemma}

\begin{proof}
We prove the stronger statement by induction on height: for monomials in
$\mcL_{\le h}$, any two are comparable, and any nonzero finite real linear
combination of distinct monomials in $\mcL_{\le h}$ which tend to $+\infty$ has
a unique dominant term. In particular, such a nonzero finite sum tends either to
$+\infty$ or to $-\infty$.

For $h=0$, this is the usual lexicographic comparison of power-log monomials.
Indeed, if
$$
P(x)=\prod_{j=0}^r \ell_j(x)^{\alpha_j},
\qquad
Q(x)=\prod_{j=0}^r \ell_j(x)^{\beta_j},
$$
then
$$
\frac{P(x)}{Q(x)}
=
\prod_{j=0}^r \ell_j(x)^{\alpha_j-\beta_j}.
$$
The first nonzero exponent difference determines whether this ratio tends to
$0$ or $+\infty$. If all exponent differences vanish, then $P=Q$. Hence
distinct power-log monomials are totally ordered by growth, and any finite
nonzero sum of distinct power-log monomials has a unique dominant term.

Now assume the statement has been proved up to height $h$. Let
$L_1,L_2\in\mcL_{\le h+1}$. We may write them in reduced form as
$$
L_i(x)=\exp(\Phi_i(x))P_i(x),
\qquad i=1,2,
$$
where each $\Phi_i$ is a finite real linear combination of monomials from
$\mcL_{\le h}$ which tend to $+\infty$, and $P_i\in\mcP$. Then
$$
\log\frac{L_1(x)}{L_2(x)}
=
\Phi_1(x)-\Phi_2(x)+\log\frac{P_1(x)}{P_2(x)}.
$$
Here $\Phi_1-\Phi_2$ is, after combining equal terms, a finite real linear
combination of distinct monomials from $\mcL_{\le h}$ which tend to $+\infty$.
Also
$$
\log\frac{P_1(x)}{P_2(x)}
$$
is a finite real linear combination of the iterated logarithms $\ell_j(x)$ for
$j\ge 1$. These logarithmic terms belong to $\mcL_0$, tend to $+\infty$, and do
not overlap with the terms appearing in the reduced exponents, because iterated
logarithms $\ell_j$ with $j\ge 1$ were excluded from reduced exponents.

Thus, after combining equal terms, the whole logarithmic ratio is either zero as
an eventual germ or a finite real linear combination of distinct monomials in
$\mcL_{\le h}$ which tend to $+\infty$. By the induction hypothesis, in the
nonzero case it has a unique dominant term. If the dominant term has positive
coefficient, then $\log\frac{L_1(x)}{L_2(x)}\to+\infty,$ and hence
$$
\frac{L_1(x)}{L_2(x)}\to+\infty,
$$
so $L_2\prec L_1$. If the dominant term has negative coefficient, then
$\log\frac{L_1(x)}{L_2(x)}\to-\infty,$ and hence
$$
\frac{L_1(x)}{L_2(x)}\to0,
$$
so $L_1\prec L_2$. The remaining case is that the logarithmic ratio is zero as an eventual germ.
Then $\log\frac{L_1(x)}{L_2(x)}=0$ eventually, so
$$
\frac{L_1(x)}{L_2(x)}=1
$$
eventually. Hence
$$
L_1=L_2
$$
as eventual germs.

Thus any two monomials in $\mcL_{\le h+1}$ are comparable. Once this
comparability is known, any finite nonzero real linear combination of distinct
monomials in $\mcL_{\le h+1}$ has a unique largest monomial, and that largest
monomial is its dominant term. This completes the induction.

Therefore, for all $L_1,L_2\in\mcL$, exactly one of
$$
L_1\prec L_2,\qquad L_2\prec L_1,\qquad L_1=L_2
$$
holds.
\end{proof}

Thus $\mcL$ is a reduced scale of positive log-exp monomial germs. In particular,
for $L_1,L_2\in\mcL$, exactly one of the following holds:
$$
L_1\prec L_2,\qquad L_2\prec L_1,\qquad L_1=L_2.
$$

We now isolate the lower-order log-exp factors which are allowed to multiply a gamma factor.

\begin{defn}
Let
$$
\mcL_{<x}:=\{L\in\mcL:\log L(x)=o(x)\qquad (x\to+\infty)\}.
$$
\end{defn}

We next enlarge the class by allowing one gamma factor. Before doing so, we need to avoid a duplication problem. By Stirling's formula, a gamma-type monomial has a leading log-exp shadow. If both the gamma factor and its pure log-exp shadow were allowed as independent growth monomials, then the same leading growth could appear in two different forms. The following definition excludes these resonant pure log-exp monomials before gamma-type monomials are added.

\begin{defn}
A pure log-exp monomial $L\in\mcL$ is called \emph{gamma-resonant} if there exist
$$
\lambda\in\RR,\qquad H\in\mcL_{<x},\qquad \kappa>0
$$
such that
$$
\frac{L(x)}{e^{\lambda x}H(x)\Gamma(\kappa x+1)}\to C
\qquad (x\to+\infty)
$$
for some $C\in\RR_{>0}$.
\end{defn}

We now adjoin a single gamma factor. This is broad enough for the principal examples in the paper, while avoiding the exact regrouping failures that occur for general products of shifted gamma factors.

\begin{defn}
A \emph{growth monomial} is either a pure log-exp germ of the form
$$
M(x)=L(x),
$$
where $L\in\mcL$ is not gamma-resonant, or a gamma-type germ of the form
$$
M(x)=e^{\lambda x}H(x)\Gamma(\kappa x+1),
$$
where
$$
\lambda\in\RR,\qquad H\in\mcL_{<x},\qquad \kappa>0.
$$
\end{defn}

The next lemma records the Stirling form used to compare gamma-type monomials with pure log-exp monomials.

\begin{lemma} \label{stirling form}
For each $\kappa>0$ there exists a constant
$C_{\kappa}\in\mathbb R_{>0}$ such that
$$
\Gamma(\kappa x+1)
=
C_{\kappa}\, \exp\!\bigl(\kappa x\log x+(\kappa\log\kappa-\kappa)x\bigr)\, x^{1/2}\,(1+o(1))
$$
as $x\to+\infty$.
\end{lemma}

\begin{proof}
Stirling's formula gives
$$
\Gamma(\kappa x+1)
=
\sqrt{2\pi}\, e^{-\kappa x}(\kappa x)^{\kappa x+1/2}(1+o(1))
\qquad (x\to+\infty),
$$
which is exactly the stated formula.
\end{proof}

We also need to know that the leading log-exp shadow of a gamma-type monomial lies on the pure log-exp side of the scale.

\begin{lemma}[Stirling shadow lies in $\mcL$]\label{lem:stirling shadow in L}
Let $H\in\mcL_{<x}$, $\lambda\in\RR$, and $\kappa>0$. Then the expression
$$
e^{\lambda x}H(x)
\exp\!\bigl(\kappa x\log x+(\kappa\log\kappa-\kappa)x\bigr)x^{1/2}
$$
has the prescribed recursive form and therefore determines an element of $\mcL$.
\end{lemma}

\begin{proof}
If $H\in\mcL_0$, then $H=P$ for some $P\in\mcP$, and we take $\Psi=0$. If
$H\in\mcL_h$ with $h\ge 1$, then we write
$$
H(x)=\exp(\Psi(x))P(x),
$$
where $P\in\mcP$ and $\Psi$ is a reduced exponent of height at most $h-1$. In either case,
$$
e^{\lambda x}H(x)
\exp\!\bigl(\kappa x\log x+(\kappa\log\kappa-\kappa)x\bigr)x^{1/2}
=
\exp(\Phi(x))\,P(x)x^{1/2},
$$
where
$$
\Phi(x)=\Psi(x)+\kappa x\log x+(\lambda+\kappa\log\kappa-\kappa)x.
$$
After combining equal monomials and deleting zero coefficients, $\Phi$ is a reduced exponent of finite height. The terms $x\log x$ and $x$ belong to $\mcL_0$, tend to $+\infty$, and neither is an iterated logarithm $\ell_j$ for $j\ge 1$. Also $P(x)x^{1/2}\in\mcP$. Hence the displayed expression has the prescribed recursive form and determines an element of $\mcL$.
\end{proof}

Using these two lemmas, we can prove trichotomy for growth monomials.

\begin{lemma}\label{Trichotomy of Growth Scales}
Let $M_1$ and $M_2$ be growth monomials. Then exactly one of the following holds:
$$
M_1\prec M_2,\qquad M_2\prec M_1,\qquad M_1=M_2 \text{ as eventual germs}.
$$
\end{lemma}

\begin{proof}
There are three cases.

First suppose both $M_1$ and $M_2$ are pure log-exp growth monomials. Then the claim follows from the trichotomy built into $\mcL$.

Now suppose both monomials are gamma-type. Write
$$
M_i(x)=e^{\lambda_i x}H_i(x)\Gamma(\kappa_i x+1)
\qquad (i=1,2).
$$
By Lemma~\ref{stirling form},
$$
\Gamma(\kappa_i x+1)
=
C_i\,\exp\!\bigl(\kappa_i x\log x+(\kappa_i\log\kappa_i-\kappa_i)x\bigr)\,x^{1/2}(1+o(1)).
$$
Hence
$$
\frac{M_2(x)}{M_1(x)}
=
\frac{H_2(x)}{H_1(x)}
\cdot
\frac{C_2}{C_1}
\cdot
\exp\!\bigl((\kappa_2-\kappa_1)x\log x+
(\lambda_2-\lambda_1+\kappa_2\log\kappa_2-\kappa_2-\kappa_1\log\kappa_1+\kappa_1)x\bigr)
\cdot
(1+o(1)).
$$
Since $H_1,H_2\in\mcL_{<x}$, we have
$$
\log\left(\frac{H_2(x)}{H_1(x)}\right)=o(x).
$$
Thus, if $\kappa_2\neq \kappa_1$, the $x\log x$ term forces either $M_1\prec M_2$ or $M_2\prec M_1$. Suppose $\kappa_2=\kappa_1$. Then the gamma factors cancel exactly, and
$$
\frac{M_2(x)}{M_1(x)}
=
e^{(\lambda_2-\lambda_1)x}\frac{H_2(x)}{H_1(x)}.
$$
If $\lambda_2\neq \lambda_1$, then the exponential factor forces either $M_1\prec M_2$ or $M_2\prec M_1$. If $\lambda_2=\lambda_1$, then the comparison reduces to
$$
\frac{H_2(x)}{H_1(x)}.
$$
Since $H_1,H_2\in\mcL$, exactly one of
$$
H_1\prec H_2,\qquad H_2\prec H_1,\qquad H_1=H_2
$$
holds. Hence exactly one of
$$
M_1\prec M_2,\qquad M_2\prec M_1,\qquad M_1=M_2
$$
holds.

Finally suppose one monomial is pure log-exp and the other is gamma-type. Without loss of generality, write
$$
M_1(x)=L(x),
\qquad
M_2(x)=e^{\lambda x}H(x)\Gamma(\kappa x+1).
$$
By Lemma~\ref{stirling form}, $M_2$ is asymptotic to a positive constant times the pure log-exp expression
$$
S(x):=e^{\lambda x}H(x)\exp\!\bigl(\kappa x\log x+(\kappa\log\kappa-\kappa)x\bigr)x^{1/2}.
$$
By Lemma~\ref{lem:stirling shadow in L}, the displayed expression determines an element $S_0\in\mcL$. By the trichotomy in $\mcL$, exactly one of
$$
L\prec S_0,\qquad S_0\prec L,\qquad L=S_0
$$
holds. In the first two cases, since $M_2$ is asymptotic to a positive constant times $S_0$, we get $M_1\prec M_2$ or $M_2\prec M_1$. In the last case, $L/M_2$ tends to a positive finite constant, so $L$ is gamma-resonant, contradicting the assumption that $L$ is an allowed pure log-exp growth monomial. Hence one of the two domination alternatives must hold.
\end{proof}

We can now define the coefficient class used for tail differences. Each coefficient germ is a finite sum of growth blocks. A growth block has a growth monomial $M(x)$ and two possible phases: a non-oscillatory part and an alternating part. The alternating factor is written as $e^{i\pi x}$ because on integer arguments it becomes $(-1)^N$, which is the oscillation appearing in the main examples.

\begin{defn}
A \emph{basic term} is a germ of the form
$$
T(x)=A\,\sigma(x)\,M(x)
$$
where $A\in\CC^\times$, $ \sigma(x)\in\{1,e^{i\pi x}\}$, and $M(x)$ is a growth monomial.
\end{defn}

\begin{defn}
For growth monomials $M_1,M_2$, write $M_1\prec M_2$ if and only if
$$
\frac{M_1(x)}{M_2(x)}\to 0
\qquad (x\to+\infty).
$$
\end{defn}

\begin{lemma}\label{Dominance of Distinct Growth Scales}
Let $M_1$ and $M_2$ be distinct growth monomials with $M_2\prec M_1$.
Then
$$
\frac{M_2(x)}{M_1(x)}\to 0
\qquad (x\to+\infty).
$$
Hence for any nonzero constants $A_1,A_2$ and any phases $\sigma_1,\sigma_2\in\{1,e^{i\pi x}\}$,
$$
\frac{A_2\sigma_2(x)M_2(x)}{A_1\sigma_1(x)M_1(x)}\to 0
\qquad (x\to+\infty).
$$
\end{lemma}

\begin{proof}
The first statement is just the definition of $\prec$. The second follows from
$$
\left|\frac{\sigma_2(x)}{\sigma_1(x)}\right|=1.
$$
\end{proof}

\begin{defn}\label{Phase blocks}
A \emph{phase block} of growth scale $M$ is an expression of the form
$$
B(x)=A\,M(x)+B\,e^{i\pi x}M(x)
$$
with $A,B\in\CC$.
\end{defn}

\begin{defn}
Let $\mcM$ denote the class consisting of the zero germ and all exact finite sums
$$
f(x)=\sum_{j=1}^r \left(A_j+B_j e^{i\pi x}\right)M_j(x)
$$
where
$$
M_1\succ M_2\succ\cdots\succ M_r
$$
are growth monomials and for each $j$ the pair $(A_j,B_j)$ is not simultaneously zero.
\end{defn}

By the trichotomy of growth scales, the ordered phase-block representation of an element of $\mcM$ is unique after combining equal growth scales and deleting zero blocks. In particular, for a fixed growth scale $M$, the coefficients $A$ and $B$ in  
$$A M(x)+B e^{i\pi x}M(x)$$ are unique, since $M(x)$ is eventually nonzero and $1$ and $e^{i\pi x}$ are linearly independent as germs on the positive real axis. Thus an element of $\mcM$ is an exact finite sum of phase blocks, one for each growth scale, with the zero germ also allowed.

The next few lemmas record the algebraic consequences of this finite ordered form. They are used later to show that a tail-difference sequence has at most one coefficient germ in $\mcM$.

\begin{lemma}\label{Closure under subtraction}
If $f,g\in\mcM$, then $f-g\in\mcM$.
\end{lemma}

\begin{proof}
If either $f$ or $g$ is the zero germ, the claim is immediate. Otherwise write
$$
f(x)=\sum_{j=1}^r \left(A_j+B_j e^{i\pi x}\right)M_j(x),
\qquad
g(x)=\sum_{k=1}^s \left(C_k+D_k e^{i\pi x}\right)N_k(x)
$$
with
$$
M_1\succ\cdots\succ M_r,
\qquad
N_1\succ\cdots\succ N_s.
$$
Take the finite union of the supports
$$
S=\{M_1,\dots,M_r\}\cup\{N_1,\dots,N_s\}.
$$
By Lemma~\ref{Trichotomy of Growth Scales}, equal growth scales are equal as eventual germs. Hence $S$ is a finite set of distinct growth monomials, which may be ordered decreasingly:
$$
P_1\succ P_2\succ\cdots\succ P_t.
$$
Extend coefficients by zero off their supports. Then
$$
f(x)-g(x)=\sum_{j=1}^t \left((a_j-c_j)+(b_j-d_j)e^{i\pi x}\right)P_j(x).
$$
Delete the indices for which both coefficients vanish. The remaining expression is either zero or an exact finite sum in the defining form for $\mcM$. Since the zero germ is included in $\mcM$, we get $f-g\in\mcM$.
\end{proof}

\begin{cor}\label{M finite linear combinations}
The class $\mcM$ is closed under finite linear combinations.
\end{cor}

\begin{proof}
Scalar multiplication is immediate from the definition, with the case of scalar $0$ giving the zero germ. Addition follows from subtraction closure, since $f+g=f-(-g)$.
\end{proof}

\begin{lemma}\label{Leading Term Form}
Let
$$
f(x)=\sum_{j=1}^r \left(A_j+B_j e^{i\pi x}\right)M_j(x)\in\mcM
$$
be nonzero, with
$$
M_1\succ M_2\succ\cdots\succ M_r.
$$
Then
$$
f(x)=\left(A_1+B_1e^{i\pi x}\right)M_1(x)+o(M_1(x))
\qquad (x\to+\infty).
$$
\end{lemma}

\begin{proof}
By the dominance lemma, $\frac{M_j(x)}{M_1(x)}\to 0$ for $j\ge 2$. Hence,
$$
\sum_{j=2}^r \left(A_j+B_j e^{i\pi x}\right)M_j(x)=o(M_1(x)).
$$
\end{proof}

\begin{lemma}\label{lemma:integer uniqueness}
If $f,g\in\mcM$ and
$$
f(n)=g(n)
$$
for all sufficiently large integers $n$, then $f=g$ as eventual germs.
\end{lemma}

\begin{proof}
By the subtraction lemma, $h:=f-g$ belongs to $\mcM$. If $h=0$, there is nothing to prove. So suppose $h\neq 0$.

Write
$$
h(x)=\sum_{j=1}^r \left(A_j+B_j e^{i\pi x}\right)M_j(x)
$$
with
$$
M_1\succ M_2\succ\cdots\succ M_r.
$$
Then by Lemma~\ref{Leading Term Form},
$$
h(x)=\left(A_1+B_1e^{i\pi x}\right)M_1(x)+o(M_1(x)).
$$

Now evaluate on even integers. Since $e^{i\pi n}=1$ when $n$ is even,
$$
h(n)=(A_1+B_1)M_1(n)+o(M_1(n))
$$
along the even subsequence. Since $h(n)=0$ for all sufficiently large integers, we must have $A_1+B_1=0$.

Similarly, on odd integers $e^{i\pi n}=-1$, so
$$
h(n)=(A_1-B_1)M_1(n)+o(M_1(n))
$$
along the odd subsequence. Again $h(n)=0$ for all sufficiently large integers, so $A_1-B_1=0$.

Hence $A_1=B_1=0$, contradicting the assumption that the leading phase block is nonzero. Therefore $h=0$.
\end{proof}

\begin{cor}
If $f\in\mcM$ and $f(n)=0$ for all sufficiently large integers $n$, then $f=0$ as an eventual germ.
\end{cor}

\begin{proof}
Apply Lemma~\ref{lemma:integer uniqueness} with $g=0$.
\end{proof}

At this point the coefficient-side ambiguity is controlled: if the tail-difference sequence comes from a coefficient germ in $\mcM$, then that germ is uniquely determined.

\subsection{Eventual sequences and shift-nonresonance}

The next condition is used when comparing the finite part with an ordinary limit. Suppose the normalized tail $R_N$ converges. Then $R_N-R_{N+1}\to 0$. Since the selector will satisfy
$$
U_c(N)-U_c(N+1)=R_N-R_{N+1},
$$
we need a condition ensuring that a selector with small first difference tends to zero. Shift-nonresonance is the condition used for this step.

\begin{defn}
For a sequence $X=(X_N)_{N\ge 0}$, define
$$
(\Delta X)_N:=X_N-X_{N+1}.
$$
\end{defn}

\begin{defn}[Eventual sequences]
Two sequences $(U_N)$ and $(V_N)$ are identified if they agree for all
sufficiently large $N$. An \emph{eventual sequence} is an equivalence class for
this relation.
\end{defn}

\begin{defn}[Shift-nonresonance]
An eventual sequence $U=(U_N)$ is called \emph{shift-nonresonant} if either
$$
U_N \to 0
$$
or there exist constants $\varepsilon>0$ and $N_0$ such that
$$
|U_N-U_{N+1}|\ge \varepsilon |U_N|
\qquad (N\ge N_0).
$$
Equivalently for the second condition, whenever $U_N\neq 0$,
$$
\left|1-\frac{U_{N+1}}{U_N}\right|\ge \varepsilon
\qquad (N\ge N_0).
$$
\end{defn}

This condition is not automatic. It will be included as part of admissibility. The following lemma explains why it is useful.

\begin{lemma}\label{lem:shift nonresonant goes to zero}
If $U$ is shift-nonresonant and
$$
(\Delta U)_N\to 0,
$$
then
$$
U_N\to 0.
$$
\end{lemma}

\begin{proof}
If $U_N$ is shift-nonresonant by the first condition, $U_N \to 0$ by definition. If $U_N$ is shift-nonresonant by the second condition, then we have
$$
|U_N|
\le \varepsilon^{-1}|U_N-U_{N+1}|
=
\varepsilon^{-1}|(\Delta U)_N|
$$
for all sufficiently large $N$. Since $(\Delta U)_N\to 0$, the result follows.
\end{proof}

\subsection{Phase splitting and canonical continuation}

We now separate the two allowed phases. The phase-free class contains the growth part only. A general element of $\mcM$ splits into a non-oscillatory part and an alternating part.

\begin{defn}[Phase-free class]
Let $\mcM^\flat$ denote the class consisting of the zero germ and all exact finite sums
$$
a(x)=\sum_{j=1}^r C_j M_j(x)
$$
where
$$
M_1\succ M_2\succ\cdots\succ M_r
$$
are growth monomials and $C_j\in\CC^\times$.
\end{defn}

\begin{lemma}[Canonical phase splitting]
Every $c\in\mcM$ admits a unique decomposition
$$
c(x)=c^{(+)}(x)+e^{i\pi x}c^{(-)}(x)
$$
with $c^{(+)},c^{(-)}\in\mcM^\flat$.
\end{lemma}

\begin{proof}
If $c=0$, take $c^{(+)}=c^{(-)}=0$. Otherwise write $c$ in its unique ordered phase-block form
$$
c(x)=\sum_{j=1}^r \left(A_j+B_j e^{i\pi x}\right)M_j(x)
$$
and set
$$
c^{(+)}(x):=\sum_{j=1}^r A_j M_j(x),
\qquad
c^{(-)}(x):=\sum_{j=1}^r B_j M_j(x).
$$
After deleting zero coefficient terms, these are elements of $\mcM^\flat$, with the zero germ allowed. Existence is immediate, and uniqueness follows from uniqueness of the ordered phase-block decomposition in $\mcM$.
\end{proof}

We now define a canonical analytic continuation of every growth monomial. Since
each element of $\mcL$ carries a chosen reduced expression, we use that expression
for analytic continuation. Namely, we replace $x$ by $z$ and use principal
branches for the iterated logarithms and powers occurring in the chosen reduced
expression.

\begin{defn}[Canonical continuation of growth monomials]
If $M(x)=L(x)$ is a pure log-exp growth monomial, define
$$
M(z):=L(z)
$$
on some right half-plane $\Re z>\sigma_M$, where $L(z)$ is obtained from the
chosen reduced expression $L$ by replacing $x$ by $z$ and taking principal branches
in the iterated logarithms and powers occurring in $L$.

If
$$
M(x)=e^{\lambda x}H(x)\Gamma(\kappa x+1)
$$
is a gamma-type growth monomial, define
$$
M(z):=e^{\lambda z}H(z)\Gamma(\kappa z+1)
$$
on some right half-plane $\Re z>\sigma_M$, where $H(z)$ is obtained from the
chosen reduced expression $H$ by replacing $x$ by $z$ and taking principal branches
in the iterated logarithms and powers occurring in $H$.
\end{defn}

\begin{defn}[Canonical continuation of phase-free germs]
If $a=0$, define its canonical continuation to be the zero function on every right half-plane. If
$$
a(x)=\sum_{j=1}^r C_j M_j(x)\in\mcM^\flat
$$
is nonzero, define its \emph{canonical continuation} on the half-plane
$$
\Re z>\sigma_a:=\max_j \sigma_{M_j}
$$
by
$$
a(z):=\sum_{j=1}^r C_j M_j(z).
$$
\end{defn}

\begin{rem}
This continuation is canonical because the ordered decomposition in $\mcM^\flat$
is exact and unique, and each growth monomial is continued from its chosen reduced
expression, together with the standard meromorphic continuation of $\Gamma$ in
the gamma-type case.
\end{rem}

\subsection{Canonical Bromwich transforms}

The selector will be defined from the canonical continuation by Bromwich inversion. Not every element of $\mcM^\flat$ is assumed to have such a transform. Bromwich-admissibility is an extra condition.

\begin{defn}[Bromwich-admissible phase-free germs]
Let $a\in\mcM^\flat$, with canonical continuation $a(z)$. We say that $a$ is
\emph{Bromwich-admissible} if there exists $\sigma_0\in\mathbb R$ such that for
every $\sigma>\sigma_0$:

\begin{enumerate}
\item the integral
$$
\mcB[a](u):=
\frac{1}{2\pi i}\int_{\sigma-i\infty}^{\sigma+i\infty}
a(z)e^{-uz}\,dz
$$
converges absolutely for almost every $u\in\mathbb R$

\item $\mcB[a](u)$ is independent of the choice of $\sigma>\sigma_0$

\item for every sufficiently large integer $N$, the integral
$$
\int_{-\infty}^{\infty}\mcB[a](u)e^{Nu}\,du
$$
converges absolutely and satisfies
$$
a(N)=\int_{-\infty}^{\infty}\mcB[a](u)e^{Nu}\,du.
$$
\end{enumerate}
\end{defn}

\begin{rem}
No extra datum is chosen here. The transform $\mcB[a]$ is defined directly from
the canonical continuation of $a$ by the standard Bromwich contour.
\end{rem}

\begin{rem}
The zero germ is Bromwich-admissible, with
$$
\mcB[0]=0.
$$
The corresponding selector operators are
$$
\mcS_+[0]=0,
\qquad
\mcS_-[0]=0.
$$
\end{rem}

The following linearity statements are used only when the relevant transforms exist.

\begin{lemma}[Linearity of the Bromwich transform]\label{lem:B-linearity}
If $a,b,a+b\in\mcM^\flat$ are Bromwich-admissible, then
$$
\mcB[a+b]=\mcB[a]+\mcB[b]
$$
almost everywhere on $\mathbb R$. If $a$ is Bromwich-admissible and $\lambda\in\CC$, then $\lambda a$ is Bromwich-admissible and
$$
\mcB[\lambda a]=\lambda\,\mcB[a]
$$
almost everywhere on $\mathbb R$.
\end{lemma}

\begin{proof}
This follows immediately from linearity of the contour integral. The scalar case is included because all defining conditions for Bromwich-admissibility are preserved under scalar multiplication.
\end{proof}

\subsection{Canonical selector operators}

We now invert the two first-order operators which appear in the phase split. The non-oscillatory part uses the kernel $(1-e^u)^{-1}$, and the alternating part uses $(1+e^u)^{-1}$.

Unless otherwise stated, improper integrals over an interval are understood to converge absolutely. If $PV\int_I F(t)\,dt$ is written and $F$ has finitely many real simple poles $(p_1,\dots,p_m)$ in $I$, then
$$
PV\int_I F(t)\,dt
:=
\lim_{\varepsilon\to0^+}
\int_{\{t\in I:\ |t-p_j|>\varepsilon\text{ for all }j\}}F(t)\,dt,
$$
provided the integral converges absolutely away from the poles and the displayed limit exists.

\begin{defn}[Selector operators on phase-free germs]
Let $a\in\mcM^\flat$ be Bromwich-admissible.

The \emph{non-oscillatory selector operator} is
$$
\mcS_+[a](N)
:=
PV\int_{-\infty}^{\infty}
\frac{\mcB[a](u)e^{Nu}}{1-e^u}\,du
$$
whenever the integral is absolutely convergent away from $u=0$ and the above
principal value exists.

The \emph{alternating selector operator} is
$$
\mcS_-[a](N)
:=
\int_{-\infty}^{\infty}
\frac{\mcB[a](u)e^{Nu}}{1+e^u}\,du
$$
whenever this improper integral converges absolutely.
\end{defn}

\begin{rem}
The only real singularity in the non-oscillatory kernel occurs at $u=0$, where
$1-e^u=0$. This is why $\mcS_+$ is defined by principal value. The alternating
kernel has no real singularity.
\end{rem}

The next two propositions show that these selectors really invert the two difference operators.

\begin{prop}[Exact inversion of $1-S$]\label{prop:exact inversion plus}
Let $a\in\mcM^\flat$ be Bromwich-admissible, and suppose $\mcS_+[a](N)$ and $\mcS_+[a](N+1)$ both exist for all sufficiently large integers $N$. Then for all sufficiently large integers $N$, 
$$
\mcS_+[a](N)-\mcS_+[a](N+1)=a(N).
$$
\end{prop}

\begin{proof}
For each $\varepsilon>0$,
$$
\begin{aligned}
&
\int_{-\infty}^{-\varepsilon}
\left(
\frac{\mcB[a](u)e^{Nu}}{1-e^u}
-
\frac{\mcB[a](u)e^{(N+1)u}}{1-e^u}
\right)\,du
+
\int_{\varepsilon}^{\infty}
\left(
\frac{\mcB[a](u)e^{Nu}}{1-e^u}
-
\frac{\mcB[a](u)e^{(N+1)u}}{1-e^u}
\right)\,du
\\
&=
\int_{-\infty}^{-\varepsilon}\mcB[a](u)e^{Nu}\,du
+
\int_{\varepsilon}^{\infty}\mcB[a](u)e^{Nu}\,du.
\end{aligned}
$$
Taking the principal-value limit as $\varepsilon\to0^+$ gives
$$
\mcS_+[a](N)-\mcS_+[a](N+1)
=
\int_{-\infty}^{\infty}\mcB[a](u)e^{Nu}\,du
=
a(N)
$$
by Bromwich-admissibility.
\end{proof}

\begin{prop}[Exact inversion of $1+S$]\label{prop:exact inversion minus}
Let $a\in\mcM^\flat$ be Bromwich-admissible, and suppose $\mcS_-[a](N)$ and
$\mcS_-[a](N+1)$ both exist for all sufficiently large integers $N$. Then for all sufficiently large integers $N$, 
$$
\mcS_-[a](N)+\mcS_-[a](N+1)=a(N).
$$
\end{prop}

\begin{proof}
By definition,
$$
\mcS_-[a](N)+\mcS_-[a](N+1)
=
\int_{-\infty}^{\infty}
\left(
\frac{\mcB[a](u)e^{Nu}}{1+e^u}
+
\frac{\mcB[a](u)e^{(N+1)u}}{1+e^u}
\right)\,du.
$$
Since
$$
\frac{e^{Nu}}{1+e^u}+\frac{e^{(N+1)u}}{1+e^u}=e^{Nu},
$$
we obtain
$$
\mcS_-[a](N)+\mcS_-[a](N+1)
=
\int_{-\infty}^{\infty}\mcB[a](u)e^{Nu}\,du
=
a(N)
$$
again by Bromwich-admissibility.
\end{proof}

\begin{cor}[Linearity of selector operators]\label{cor:S-linearity}
Let $a,b,a+b\in\mcM^\flat$ be Bromwich-admissible. Whenever all selector integrals involved exist,
$$
\mcS_\pm[a+b]=\mcS_\pm[a]+\mcS_\pm[b].
$$
If $a$ is Bromwich-admissible and $\lambda\in\CC$, then, whenever the selector integrals for $a$ exist,
$$
\mcS_\pm[\lambda a]=\lambda\,\mcS_\pm[a].
$$
\end{cor}

\begin{proof}
This follows from Lemma~\ref{lem:B-linearity} and linearity of the defining integrals.
\end{proof}

\subsection{Canonical selectors}

We now combine the two phase-free selectors. The non-oscillatory part is inverted by $\mcS_+$, while the alternating part is inverted by $\mcS_-$ and then multiplied by $(-1)^N$ on integer arguments.

\begin{defn}[Canonical selector attached to a coefficient germ]
Let $c\in\mcM$, with phase split
$$
c(x)=c^{(+)}(x)+e^{i\pi x}c^{(-)}(x),
\qquad
c^{(+)},c^{(-)}\in\mcM^\flat.
$$
Assume that $c^{(+)}$ and $c^{(-)}$ are Bromwich-admissible and that the
selector integrals $\mcS_+[c^{(+)}](N)$ and $\mcS_-[c^{(-)}](N)$ exist for all
sufficiently large $N$. Define the \emph{canonical selector attached to $c$} by
$$
U_c(N):=\mcS_+[c^{(+)}](N)+(-1)^N\mcS_-[c^{(-)}](N).
$$
\end{defn}

\begin{prop}[Exact selector identity]\label{prop:Uc identity}
Let $c\in\mcM$ admit a canonical selector $U_c$. Then for all sufficiently large
$N$,
$$
U_c(N)-U_c(N+1)=c(N).
$$
\end{prop}

\begin{proof}
By Proposition~\ref{prop:exact inversion plus},
$$
\mcS_+[c^{(+)}](N)-\mcS_+[c^{(+)}](N+1)=c^{(+)}(N)
$$
and by Proposition~\ref{prop:exact inversion minus},
$$
\mcS_-[c^{(-)}](N)+\mcS_-[c^{(-)}](N+1)=c^{(-)}(N).
$$
Hence
$$
\begin{aligned}
U_c(N)-U_c(N+1)
&=
\bigl(\mcS_+[c^{(+)}](N)-\mcS_+[c^{(+)}](N+1)\bigr)
\\
&\quad
+
(-1)^N\bigl(\mcS_-[c^{(-)}](N)+\mcS_-[c^{(-)}](N+1)\bigr)
\\
&=
c^{(+)}(N)+(-1)^N c^{(-)}(N)
\\
&=
c(N).
\end{aligned}
$$
\end{proof}

We will also need the following closure and linearity statements for canonical selectors.

\begin{lemma}[Scalar closure of canonical selectors]\label{lem:selector scalar closure}
If $c\in\mcM$ admits a canonical selector and $\lambda\in\CC$, then $\lambda c$
admits a canonical selector. Moreover,
$$
U_{\lambda c}=\lambda U_c
$$
as eventual sequences.
\end{lemma}

\begin{proof}
If $\lambda=0$, then $\lambda c=0$, and the zero selector gives the claim. Assume $\lambda\neq0$. Write
$$
c=c^{(+)}+e^{i\pi x}c^{(-)}.
$$
Then
$$
(\lambda c)^{(+)}=\lambda c^{(+)},
\qquad
(\lambda c)^{(-)}=\lambda c^{(-)}.
$$
By Lemma~\ref{lem:B-linearity}, the phase-free germs $\lambda c^{(+)}$ and
$\lambda c^{(-)}$ are Bromwich-admissible, and by Corollary~\ref{cor:S-linearity},
$$
\mcS_+[\lambda c^{(+)}]=\lambda\mcS_+[c^{(+)}],
\qquad
\mcS_-[\lambda c^{(-)}]=\lambda\mcS_-[c^{(-)}].
$$
Therefore
$$
U_{\lambda c}(N)
=
\lambda\mcS_+[c^{(+)}](N)+(-1)^N\lambda\mcS_-[c^{(-)}](N)
=
\lambda U_c(N).
$$
\end{proof}

\begin{lemma}[Linearity of canonical selectors]\label{lem:canonical selector linearity}
If $c,d,c+d\in\mcM$ admit canonical selectors, then
$$
U_{c+d}=U_c+U_d
$$
as eventual sequences.
\end{lemma}

\begin{proof}
The phase split is linear:
$$
(c+d)^{(+)}=c^{(+)}+d^{(+)},
\qquad
(c+d)^{(-)}=c^{(-)}+d^{(-)},
$$
after deleting zero terms. The identity follows from Corollary~\ref{cor:S-linearity}.
\end{proof}

\subsection{$\mcM$-admissibility and the finite part}

Let $R=(R_N)_{N\ge 0}$ be a normalized tail and define
$$
c_N:=R_N-R_{N+1}.
$$

The next definition says when the tail difference is structured enough for the finite part to be defined. Membership in $\mcM$ is only the first requirement. We also require a canonical selector and shift-nonresonance.

\begin{defn}[$\mcM$-admissibility]
We say that $R$ is \emph{$\mcM$-admissible} if either

\begin{enumerate}
\item $c_N=0$ for all sufficiently large $N$

\item or there exists a coefficient germ $c\in\mcM$ such that
$$
c(N)=c_N
$$
for all sufficiently large integers $N$, this germ admits a canonical selector
$U_c$, and $U_c$ is shift-nonresonant.
\end{enumerate}
\end{defn}

By Lemma~\ref{lemma:integer uniqueness}, the coefficient germ $c$ in the second case is unique if it exists. Thus uniqueness is not an extra assumption; it follows from the coefficient class.

\begin{cor}[Representation invariance of the coefficient germ]
Whenever the tail-difference sequence extends to a germ in $\mcM$, that germ is
unique.
\end{cor}

\begin{proof}
This is exactly Lemma~\ref{lemma:integer uniqueness}.
\end{proof}

\begin{cor}[Representation invariance of the selector]
Whenever the tail-difference sequence extends to a germ in $\mcM$ admitting a
canonical selector, that selector is uniquely determined by the eventual sequence
$c_N$.
\end{cor}

\begin{proof}
The germ in $\mcM$ is unique, its phase split is unique, the canonical
continuation is determined by the chosen reduced expressions in $\mcL$, the
Bromwich transform is canonical, and the selector formulas are explicit.
\end{proof}

We can now define the finite part. If the tail difference is eventually zero, the tail is eventually constant. Otherwise, we subtract the canonical selector. The selector has the same first difference as the tail, so the difference between the tail and the selector is eventually constant.

\begin{defn}[Finite part]
Let $R$ be $\mcM$-admissible. If $c_N=0$ eventually, define $\FP(R)$ to be the
eventual constant value of $R_N$. Otherwise let $c\in\mcM$ be the unique
tail-difference germ, and define $\FP(R)$ to be the eventual constant value of
$$
R_N-U_c(N).
$$
\end{defn}

\begin{prop}[Well-definedness of $\FP$]
The quantity $\FP(R)$ is well defined.
\end{prop}

\begin{proof}
If $c_N=0$ eventually, then $R_N$ is eventually constant, so the claim is
immediate.

Otherwise, Proposition~\ref{prop:Uc identity} gives
$$
R_N-R_{N+1}=U_c(N)-U_c(N+1)
$$
for all sufficiently large $N$. Hence
$$
(R_N-U_c(N))-(R_{N+1}-U_c(N+1))=0
$$
for all sufficiently large $N$, so $R_N-U_c(N)$ is eventually constant.
\end{proof}

\begin{cor}[Eventual invariance]\label{cor:eventual invariance}
If $R$ and $\widetilde R$ agree for all sufficiently large $N$, and $R$ is
$\mcM$-admissible, then $\widetilde R$ is $\mcM$-admissible and
$$
\FP(\widetilde R)=\FP(R).
$$
\end{cor}

\begin{proof}
Since $R_N=\widetilde R_N$ for all sufficiently large $N$, the tail-difference
sequences also agree for all sufficiently large $N$. Hence the same
tail-difference germ and the same canonical selector are used in the definition
of the finite part. Therefore $R_N-U_c(N)$ and $\widetilde R_N-U_c(N)$ have the
same eventual constant value.
\end{proof}

The following basic properties explain how the finite part behaves under the operations needed later. We first treat constant shifts.

\begin{cor}\label{constant shift}
For any constant $\lambda\in\CC$ and any $\mcM$-admissible tail $R$, the tail
$R+\lambda$ is $\mcM$-admissible and
$$
\FP(R+\lambda)=\FP(R)+\lambda.
$$
\end{cor}

\begin{proof}
If $c_N=0$ eventually, then $R_N$ is eventually constant, and the claim is immediate.

Otherwise, adding a constant does not change the tail-difference sequence:
$$
(R_N+\lambda)-(R_{N+1}+\lambda)=R_N-R_{N+1}.
$$
Hence the associated germ and canonical selector are unchanged. In particular,
$R+\lambda$ is $\mcM$-admissible. Moreover
$$
(R_N+\lambda)-U_c(N)=(R_N-U_c(N))+\lambda
$$
eventually, and the eventual constant increases by $\lambda$.
\end{proof}

\begin{cor}\label{constant scaling}
If $R$ is $\mcM$-admissible and $\lambda\in\CC$, then $\lambda R$ is
$\mcM$-admissible and
$$
\FP(\lambda R)=\lambda\,\FP(R).
$$
\end{cor}

\begin{proof}
If $\lambda=0$, then $\lambda R$ is the zero tail and the claim is immediate.

If $c_N=0$ eventually, then $R_N$ is eventually constant, so $\lambda R_N$ is
eventually constant and the claim follows.

Otherwise, the tail-difference germ changes from $c$ to $\lambda c$. By
Lemma~\ref{lem:selector scalar closure}, the germ $\lambda c$ admits the canonical
selector
$$
U_{\lambda c}=\lambda U_c.
$$
The selector $\lambda U_c$ is shift-nonresonant whenever $U_c$ is. Hence
$\lambda R$ is $\mcM$-admissible. Finally,
$$
\lambda R_N-\lambda U_c(N)=\lambda(R_N-U_c(N))
$$
eventually, and taking eventual constants gives the result.
\end{proof}

The additivity statement is not a closure theorem. It applies only when the two tails and their sum are already known to be admissible.

\begin{cor}\label{additivity}
Suppose $X$, $Y$, and $X+Y$ are $\mcM$-admissible normalized tails. Then
$$
\FP(X+Y)=\FP(X)+\FP(Y).
$$
\end{cor}

\begin{proof}
If one of $X$ or $Y$ is eventually constant, the claim follows from Corollary~\ref{constant shift}. If $X+Y$ is eventually constant, say with eventual value $\eta$, then $Y=\eta-X$ eventually, and the claim follows from Corollary~\ref{constant shift} and Corollary~\ref{constant scaling}.

Otherwise let $c_X,c_Y,c_{X+Y}$ be the tail-difference germs attached to $X,Y,X+Y$. Since
$$
(X_N+Y_N)-(X_{N+1}+Y_{N+1})=(X_N-X_{N+1})+(Y_N-Y_{N+1}),
$$
we have
$$
c_{X+Y}(N)=c_X(N)+c_Y(N)
$$
for all sufficiently large $N$. By Corollary~\ref{M finite linear combinations}, $c_X+c_Y\in\mcM$, and by uniqueness in $\mcM$,
$$
c_{X+Y}=c_X+c_Y.
$$
By Lemma~\ref{lem:canonical selector linearity},
$$
U_{X+Y}=U_X+U_Y
$$
eventually. By eventual invariance of the finite part, we may compute using this eventual representative. Therefore
$$
(X_N+Y_N)-U_{X+Y}(N)
=
(X_N-U_X(N))+(Y_N-U_Y(N))
$$
eventually, and taking eventual constants gives the result.
\end{proof}

Finally, the finite part agrees with the ordinary limit whenever the normalized tail actually converges. This is the basic regularity property of the finite-part extraction.

\begin{cor}\label{agreement with limits}
If $R$ is $\mcM$-admissible and
$$
R_N\to L,
$$
then
$$
\FP(R)=L.
$$
\end{cor}

\begin{proof}
If $c_N=0$ eventually, then $R_N$ is eventually constant and the claim is
immediate.

Otherwise
$$
U_c(N)-U_c(N+1)=c(N)=R_N-R_{N+1}\to 0
$$
and $U_c$ is shift-nonresonant by admissibility. Hence, by
Lemma~\ref{lem:shift nonresonant goes to zero},
$$
U_c(N)\to 0.
$$
Since $R_N-U_c(N)$ is eventually constant, its eventual constant value equals
its limit:
$$
\FP(R)=\lim_{N\to\infty}(R_N-U_c(N))=L.
$$
\end{proof}

\begin{rem}
    The present paper uses ordinary Bromwich transforms, so the transform $$ \mcB[a](u) = \frac{1}{2\pi i}\int_{\sigma-i\infty}^{\sigma+i\infty}a(z)e^{-uz}\,dz $$
    is required to exist as an ordinary function. There is also a natural distributional extension of the definition. For instance, one would formally have
    $$
    \mcB[e^{\lambda x}]=\delta_\lambda.
    $$
    We do not develop this extension here, since the ordinary Bromwich-admissible class already covers the main divergent examples considered in this paper.
\end{rem}

\begin{rem}
    The series treated here can also be easily extended to phases $e^{i\theta N}$ where $\theta \in [0,2\pi)$. Then, the selector operators become $$S_\theta[a](N) := PV \int_{-\infty}^\infty \frac{\mcB[a](u) e^{Nu}}{1-e^{i\theta}e^{u}} du$$  Indeed, if
    $U_N=e^{i\theta N}\mcS_\theta[a](N)$ then formally $U_N-U_{N+1}=e^{i\theta N}a(N)$.
\end{rem}

\section{C-Summation} \label{C-summation}
\begin{defn} [C-Summation]
    If there exists some solution ($I_n$) of the recurrence $a(n)I_n +b(n)I_{n+1} = \mu_n$ whose tail $R = (R_N)_{N\ge 0}$ is $\mathcal{M}$ admissible, then the associated series $$\sum_{k=0}^\infty (-1)^k{\mu_{n_0+k}} \prod_{j=0}^{k-1} \frac{b(n_0+j)}{a(n_0+j+1)}$$ is $C$-Summable to $$S={a(n_0)I_{n_0}} - \text{FP}_\infty (R)$$ 
\end{defn}

The definition above is motivated through the continued fraction method described in section $\ref{continued fractions representing certain integrals}$ and the discussion of the finite part in $\ref{finite term}$. The term $a(n_0)I_{n_0}$ records the value obtained from a chosen solution of the recurrence, while the subtraction of the finite part removes the contribution of the homogeneous part which survives in the tail. In this sense, $C$-Summation should be viewed as a renormalized initial value. 

The first question is whether the construction depends on which $I_n$ is used. Since $\ref{the reccurence ambiguity}$ showed that any two solutions differ by a homogeneous term, this is one of the most important properties for the method to posses. The role of the finite part is what absorbs this change. If the tail $R_N$ shifts by a constant, the initial value shifts by the same constant in the opposite direction, so the corrected quantity remains unchanged. The following result makes this precise. 

\begin{lemma}
    $C$-summation does not depend on the specific choice of admissible solution $I_n$.
\end{lemma}

\begin{proof}
    Suppose $I_n$ and $J_n$ both satisfy the recurrence and have $\mcM$ admissible tails. Then,  $J_n = I_n + CP(n,n_0)$, for some constant $C$. Compute $R_N^I$ and $R_N^J$: $$R_N^I = a(n_0)\frac{I_{n_0+N}}{P(n_0+N,n_0)}$$ $$R_N^J = a(n_0)\frac{J_{n_0+N}}{P(n_0 +N,n_0)} = a({n_0})\frac{I_{n_0+N} + CP(n_0+N,n_0)}{P(n_0+N,n_0)} = R_N^I + a(n_0) C$$

    By \ref{constant shift}, $\text{FP}_\infty (R^{J}) = \text{FP}_\infty(R^I + a(n_0)C) = \text{FP}_\infty(R^I) + a(n_0)C$. Then, \begin{align*} a(n_0)J_{n_0} - \text{FP}_\infty(R^J) &= a(n_0) (I_{n_0} + CP(n_0,n_0))- \text{FP}_\infty(R^I + a(n_0) C) \\ &= a(n_0)I_{n_0} + a(n_0)C - \text{FP}_\infty(R^I) - a(n_0)C \\ &= a(n_0)I_{n_0} - \text{FP}_\infty(R^I)\end{align*}
\end{proof}
The point of the lemma is that the ambiguity in the solution and the ambiguity in the normalized tail are the same ambiguity written in two different ways. Once the finite part is subtracted, the dependence on the chosen solution disappears. The next step is to record a simple identity relating the chosen initial value, the partial sums of the associated formal series, and the normalized tail. This identity will be used repeatedly in the rest of the section.
\begin{prop} \label{partial sums}
    Suppose $$\mathcal{S} = \sum_{k=0}^\infty (-1)^k{\mu_{n_0+k}} \prod_{j=0}^{k-1} \frac{b(n_0+j)}{a(n_0+j+1)}$$ is a formal series with associated sequence $I_n$. Then, for every $N\ge1$, we have that $a(n_0)I_{n_0} = S_N + R_N$, where $S_N$ is the $N$th partial sum of the series and $R_N$ is the normalized tail. 
\end{prop}
\begin{proof}
    Proceeding as in \ref{formal series}, if we take $a(n)I_n = T_n$, we have that $$T_{n_0} = \bigg(\sum_{k=0}^{N-1} (-1)^k\mu_{{n_0}+k} \prod_{j=0}^{k-1} \frac{b({n_0}+j)}{a({n_0}+j+1)} \bigg) + 
    T_{n_0+N}\prod_{j=0}^{N-1} -\frac{b({n_0}+j)}{a(n_0+j+1)}$$
    Now consider the remainder term, which is $$T_{n_0+N}\prod_{j=0}^{N-1} -\frac{b({n_0}+j)}{a(n_0+j+1)} = T_{n_0+N} \frac{a(n_0)}{a(n_0+N)} \prod_{k=n_0}^{n_0+N-1} -\frac{b(k)}{a(k)}$$ $$=a(n_0+N)I_{n_0+N}  \frac{a(n_0)}{a(n_0+N)} \frac{1}{P(n_0+N,n_0)} = \frac{a(n_0)I_{n_0+N}}{P(n_0+N,n_0)} = R_N$$ 

    Now, notice that the first term $\sum_{k=0}^{N-1} (-1)^k\mu_{{n_0}+k} \prod_{j=0}^{k-1} \frac{b({n_0}+j)}{a({n_0}+j+1)}$ is exactly the $N$th partial sum of the series. Thus, $T_{n_0} = S_N + R_N$ as desired. 
\end{proof}

This formula essentially proves that the tail $R_N$ is the actual tail of the associated series. 

There is a second notion of well-definedness that must also be addressed. A given formal series may admit more than one recurrence representation, and a summation method based on recurrences would be unacceptable if different presentations of the same series produced different values. The next result ensures that this does not happen. Once two recurrences encode the same formal coefficients, the difference between tails is once again constant, and the finite part correction removes it. Thus, the method only depends on the series, not on the auxiliary recurrence used to present it.   
\begin{lemma}
    C-Summation is well defined. Equivalently, if formal series $$\mcS=\sum_{k=0}^\infty c_k$$ admits two recurrence representations $$a(n)I_n + b(n)I_{n+1} = \mu_n \ \ \ \ \ \text{and} \ \ \ \ \ \tilde{a}(n)\tilde{I}_n + \tilde{b}(n) \tilde{I}_{n+1} = \tilde{\mu}_n$$ with $$c_k = (-1)^k \mu_{n_0+k} \prod_{j=0}^{k-1} \frac{b(n_0 + j)}{a(n_0+j+1)} = (-1)^k \tilde{\mu}_{n_0+k} \prod_{j=0}^{k-1} \frac{\tilde{b}(n_0+j)}{\tilde{a}(n_0+j+1)}$$ which both have $\mcM$ admissible tails, then the two $C$-sums induced by both of these recurrences are equal. 
\end{lemma}
\begin{proof}
    We know that $a(n_0)I_{n_0}=T_{n_0} = S_N + R_N$, where $S_N$ are the partial sums of the series $\mcS$, as we show in \ref{partial sums}. Similarly, $\tilde{a}(n_0)\tilde{I}_{n_0} = \tilde{T}_{n_0} = S_N+ \tilde{R}_N$. Subtracting the two gives $a(n_0)I_{n_0} -\tilde{a}(n_0)\tilde{I}_{n_0} = R_N - \tilde{R}_N $. The left hand side is independent of $N$, so there is a constant $C$ such that $R_N = \tilde{R}_N + C$. Now apply \ref{constant shift}. Since both $R$ and $\tilde{R}$ are $\mcM$ admissible, we have $$\text{FP}_\infty (R) = \text{FP}_\infty(\tilde{R} + C) = \text{FP}_\infty(\tilde{R}) + C$$ Thus, the difference of the $C$ sums is given by $$(a(n_0)I_{n_0}-\text{FP}_\infty({R})) -(\tilde{a}(n_0)\tilde{I}_{n_0}-\text{FP}_\infty(\tilde{R})) = C + (-C) = 0$$ This shows that the $C$-sums are the same. 
\end{proof}

Having shown that the method only depends on the series being summed, we next verify that it is a meaningful extension of summation in the convergent case. This is the regularity property. In the convergent case, the normalized tail tends to a limit, and the finite part reproduces that limit exactly. The construction therefore extends ordinary summation. 
\begin{lemma} \label{regularity proof}
    If a convergent series is $C$-Summable, then its $C$-Sum is equal to its ordinary sum. 
\end{lemma}

\begin{proof}
    We can express $T_{n_0}$ as $S_N + R_N$ for each finite $N$ where $S_N$ are the partial terms of $\mathcal{S}$. From there we get $S_N = T_{n_0} - R_N$. As $N\to \infty $, this means $S_N = T_{n_0} -R_N  \to \mathcal{S}$. Rearranging gives $R_N = T_{n_0}-S_N \to T_{n_0} - \mathcal{S}$ By \ref{agreement with limits}, this means $\text{FP}_\infty (R) = T_{n_0} - \mathcal{S}$. Finally, this means that the C-summation is $$T_{n_0} - \text{FP}_\infty (R) = T_{n_0} - (T_{n_0} - \mathcal{S}) = \mathcal{S}$$
\end{proof}

We now turn to proving properties of the method. The first is homogeneity. If every term of the input series is scaled by some constant, then the recurrence and its admissible tail are scaled by the same constant, meaning the assigned value should scale accordingly. The property confirms that $C$-Summation responds to size in the same way ordinary summation does and does not introduce any nonlinear distortion under scalar multiplication. 

\begin{lemma}
    C-Summation is homogeneous. 
\end{lemma}
\begin{proof}
    Suppose $I_n$ is a sequence with $a(n)I_n + b(n)I_{n+1} = \mu_n$, with $\mathcal{M}$ admissible tail $R$. Thus $C$-Summation is defined. Then, we have that $a(n)(c\cdot I_n) + b(n)(c\cdot I_{n+1}) = c\cdot \mu_n$. We know that $\text{FP}_\infty(cR) = c\cdot \text{FP}_\infty(R)$. Applying $C$-Summation to both sequences, we get that $$a(n_0)I_{n_0} - \text{FP}_\infty (R) = \sum_{k=0}^\infty (-1)^k{\mu_{n_0+k}} \prod_{j=0}^{k-1} \frac{b(n_0+j)}{a(n_0+j+1)} $$ $$a(n_0)(c I_{n_0})  - \text{FP}_\infty (cR) = c(a(n_0) I_{n_0} - \text{FP}_\infty (R)) $$ $$=  \sum_{k=0}^\infty (-1)^k (c\mu_{n_0+k}) \prod_{j=0}^{k-1} \frac{b(n_0+j)}{a(n_0+j+1)}$$ This shows that multiplying the sum by a constant $c$ also multiplies its $C$-sum by $c$, showing that $C$-Summation is homogeneous, as desired. 

\end{proof}

The next property is stability, which controls the effect of changing finitely many initial terms. For ordinary convergent series, deleting the first term simply subtracts that term from the sum, and one would like the same behavior here. In the recurrence setting, stability is more delicate because the assigned value is defined through the tail rather than a limit of remainders. The shift from $R_N$ to $R_{N+1}$ is therefore the critical issue. If the finite part is preserved under this step, then stability holds, and truncating the series changes the value in the expected way.  

\begin{lemma}
    $C$-Summation is stable whenever the normalized tail is shift compatible in the sense that it satisfies the property that $R$ and $R^+$ are both $\mcM$ admissible and $\text{FP}_\infty (R) = \text{FP}_\infty (R^+)$ where $R^+$ denotes the shifted tail $R_N^+ := R_{N+1}$.
\end{lemma}
\begin{proof}

    Suppose $I_n$ is a sequence with $a(n)I_n + b(n)I_{n+1} = \mu_n$, with $\mathcal{M}$ admissible tail $R=R^{(n_0)}$. Thus $C$-Summation is defined.

    Now, let $$R_N^{(n_0)} = \frac{a(n_0)I_{n_0+N}}{P(n_0+N,n_0)} \qquad R_N^{(n_0+1)} = \frac{a(n_0+1)I_{{n_0+1+N}}}{P(n_0+1+N, n_0+1)}$$

    We know that $$R_{N+1}^{(n_0)} = -\frac{b({n_0})}{a(n_0+1)}R_{N}^{(n_0+1)}$$ because \begin{align*}R_N^{(n_0+1)} =& \frac{a(n_0+1)I_{{n_0+1+N}}}{P(n_0+1+N, n_0+1)}=  \frac{a(n_0+1)I_{n_0+1+N}}{-\frac{ b(n_0)}{a(n_0)} P(n_0+N+1, n_0)} \\ =& -\frac{a(n_0)a(n_0+1)I_{n_0+N+1}}{b(n_0) P(n_0+N+1,n_0)} = - \frac{a(n_0+1)}{b(n_0)}R_{N+1}^{(n_0)}
    \end{align*}

    Now consider the $n_0+1$ shifted series. This can be rewritten so that the index starts at 0.  $$\sum_{k=1}^\infty (-1)^k{\mu_{n_0+k}} \prod_{j=0}^{k-1} \frac{b(n_0+j)}{a(n_0+j+1)} =-\frac{b(n_0)}{a(n_0+1)}\sum_{k=0}^\infty (-1)^{k}{\mu_{(n_0+1)+k}} \prod_{j=0}^{k-1} \frac{b((n_0+1)+j)}{a((n_0+1)+j+1)}$$ 
    
    Equivalently, as sequences, $(R^{(n_0)})^+ = -\frac{b(n_0)}{a(n_0+1)}R^{(n_0+1)}$. Now, we notice that the second series is in the correct form for $C$-Summation. Since $(R^{(n_0)})^+$ is $\mcM$ admissible by shift compatibility, \ref{constant scaling} implies that $R^{(n_0+1)}$ is  $\mcM$ admissible. The second series then has $C$-Sum $a(n_0+1)I_{n_0+1} - \text{FP}_\infty (R^{(n_0+1)})$. Plugging this into the above expression gives us $$=- b(n_0)I_{n_0+1}  + \frac{b(n_0)}{a(n_0+1)}\text{FP}_\infty (R^{(n_0+1)}) $$ Using the recurrence $\mu_{n_0} - a(n_0)I_{n_0} = b(n_0)I_{n_0+1}$, we get $$\mu_{n_0} + \sum_{k=1}^\infty (-1)^k{\mu_{n_0+k}} \prod_{j=0}^{k-1} \frac{b(n_0+j)}{a(n_0+j+1)} = a(n_0)I_{n_0} +\frac{b({n_0})}{a(n_0+1)} \text{FP}_\infty(R^{(n_0+1)}) $$ Now, recall that $ -\frac{a(n_0+1)}{b(n_0)}R_{N+1}^{(n_0)} =R_{N}^{(n_0+1)}$. Equivalently, as sequences, $(R^{(n_0)})^+ = -\frac{b(n_0)}{a(n_0+1)}R^{(n_0+1)}$. This, along with $\ref{constant scaling}$ gives us that $$\mu_{n_0} + \sum_{k=1}^\infty (-1)^k{\mu_{n_0+k}} \prod_{j=0}^{k-1} \frac{b(n_0+j)}{a(n_0+j+1)} = a(n_0)I_{n_0} -\text{FP}_\infty((R^{(n_0)})^+)$$ Finally, applying the shift compatibility assumption $\text{FP}_\infty(R^{(n_0)}) = \text{FP}_\infty((R^{(n_0)})^+)$ gives us that $$\mu_{n_0} + \sum_{k=1}^\infty (-1)^k{\mu_{n_0+k}} \prod_{j=0}^{k-1} \frac{b(n_0+j)}{a(n_0+j+1)} = a(n_0)I_{n_0} - \text{FP}_\infty(R^{(n_0)})$$ This is exactly stability.

\end{proof}

The final step of our proof relies on the identity $\text{FP}_\infty (R) = \text{FP}_\infty(R^+)$. This is not a consequence of $\mcM$ admissibility alone. A one step shift could change the finite part of the expansion. This is not a question of whether the summation method assigns a value. In cases where the hypothesis is false, the method may assign a value perfectly well. What can fail however, is the stability principle, namely removing finitely many initial terms only changes the sum by exactly that finite amount. 

This distinction is important. If shift-compatibility fails, that does not mean $C$-Summation becomes meaningless. In these cases truncating or re-indexing the series may alter the assigned value in a way that is not governed by the usual stability axiom. However, the results obtained may be less useful than in cases where shift-compatibility does not fail. 

The examples in this paper all satisfy stability. However, it is important to understand where stability fails and how it affects the applicability of the method. 

Finally, one can check how much linearity survives. Full linearity across arbitrary recurrence presentations is too much to expect, and is false in general. However, when $a(n)$ and $b(n)$ are fixed and only $\mu_n$ varies, the structure is much better behaved. In that case, we get linearity in the third argument, when we express the associated series as a triple $(a(n), b(n), \mu_n)$.  
\begin{lemma}
    C-Summation is linear in the third argument. In particular, suppose two series are defined by the triples $S_1=(a(n), b(n), \mu_n)$ and $S_2=(a(n), b(n), \nu_n)$, with $\mcM$ admissible tails $R^{I}$ and $R^{J}$. If $R^I$ and $R^J$ and $R^I+R^J$ are $\mcM$ admissible then the $C$-sum of $(a(n),b(n), \mu_n + \nu_n)$ is equal to the sum of $S_1$ and $S_2$.
\end{lemma}

\begin{proof}
    We can describe the associated series by the triple $(a(n), b(n), \mu_n)$. Suppose $a(n)I_n + b(n) I_{n+1} = \mu_n $  and $a(n)J_n + b(n) J_{n+1} = \nu_n$, where $I_n, J_n$ have $R^{I}$ and $R^J$  both $\mathcal{M}$ admissible. Now, we see that $a(n)(I_n+J_n)+b(n)(I_{n+1}+J_{n+1}) = \mu_n +\nu_n$. By assumption we have that $R^I+R^J$ is also $\mathcal{M}$ admissible. Then, by \ref{additivity} we know that $\text{FP}_\infty (R^I)+ \text{FP}_\infty (R^J) = \text{FP}_\infty (R^I+R^J)$ Then, we have that $$a(n_0)I_{n_0}- \text{FP}_\infty (R^I) +a(n_0)J_{n_0} - \text{FP}_\infty (R^J)$$ $$= \sum_{k=0}^\infty (-1)^k{\mu_{n_0+k}} \prod_{j=0}^{k-1} \frac{b(n_0+j)}{a(n_0+j+1)} + \sum_{k=0}^\infty (-1)^k{\nu_{n_0+k}} \prod_{j=0}^{k-1} \frac{b(n_0+j)}{a(n_0+j+1)} $$ and that $$\sum_{k=0}^\infty (-1)^k{(\mu_{n_0+k}+ \nu_{n_0+k)}} \prod_{j=0}^{k-1} \frac{b(n_0+j)}{a(n_0+j+1)}$$ $$ = a(n_0)(I_{n_0} + J_{n_0}) - \text{FP}_\infty (R^I + R^J)$$

    However, we have that $$a(n_0)I_{n_0}- \text{FP}_\infty (R^I) +a(n_0)J_{n_0} - \text{FP}_\infty (R^J)$$ $$ = a(n_0)(I_{n_0} + J_{n_0})-  \text{FP}_\infty (R^I)- \text{FP}_\infty (R^J)$$ $$= a(n_0)(I_{n_0} + J_{n_0})-  \text{FP}_\infty (R^I+R^J)$$
    This proves linearity in the third argument. 
\end{proof}

In summary, the results above show that $C$-Summation is a well defined summation method. The value it assigns does not depend on the choice of solution to the underlying recurrence, nor on the particular recurrence used to represent the series. It also agrees with ordinary summation in the convergent case and respects scalar multiplication. Under a shift compatibility hypothesis, it is also stable. The key idea is that the ambiguity coming from the homogeneous part of the recurrence is captured by the normalized tail and removed by taking its finite part. In this way, the method leads to a natural and consistent way of assigning values to divergent series while preserving the main properties one expects of a summation method.

\section{Class Theorems}

The finite-part construction gives a definition of $C$-summability once the normalized tail is known to be $\mcM$-admissible. The purpose of this section is to give concrete hypotheses which imply that admissibility. In particular, the main result below is a class theorem for rapidly growing alternating series: it shows that if a phase-free coefficient germ has sufficient vertical decay and its integer values grow with successive ratio tending to infinity, then $\sum_{n=0}^{\infty}(-1)^n f(n)$ is $C$-summable.

The proof has two components. First, a ratio criterion shows that
rapid growth of the tail difference forces the associated selector to be shift-nonresonant. Second, a vertical-line Bromwich inversion theorem gives checkable analytic conditions under which the Bromwich transform exists, reconstructs $f(N)$, and defines the alternating selector. Combining these two facts proves the alternating class theorem.

After proving this admissibility theorem, we compare the resulting $C$-sum with ordinary Borel-Laplace summation under an additional exponential-integrability hypothesis. This comparison is not meant to describe the full domain of $C$-summation; it identifies one concrete situation in which the recurrence selector and the Borel-Laplace integral compute the same value.

We begin with the growth criterion for shift-nonresonance. Suppose $U_N$ satisfies $U_N-U_{N+1}=c_N.$ If $c_N$ grows much faster than $c_{N-1}$, then the accumulated contribution from earlier differences is negligible compared with $c_N$. Thus $U_N$ is asymptotically smaller than its own first difference. This forces the second alternative in the definition of shift-nonresonance.

\begin{lemma} \label{shift sufficient ratio}
    Suppose $c_N \neq 0$ for all sufficiently large $N$, and $$\frac{|c_{N}|}{|c_{N-1}|} \to \infty$$ Then, let $U_N$ be any eventual sequence satisfying $U_N - U_{N+1} = c_N$ for sufficiently large $N$. Then $U$ is shift-nonresonant. 
\end{lemma}
\begin{proof}
    Choose $q$ such that $0<q<1$. Since $\frac{|c_N|}{|c_{N-1}|} \to \infty$, we also have $\frac{|c_{N-1}|}{|c_N|} \to 0$. Then, for sufficiently large $j\ge N_0$, we have $|c_j|\leq q|c_{j+1}|$. Iterating this inequality gives $|c_j|\leq q^{N-j - 1}|c_{N-1}| $. Now consider the recurrence $U_N - U_{N+1} = c_N$. Starting at $N=N_0$ and iterating gives $$U_{N_0}- \sum_{j=N_0}^{N-1} c_j = U_N$$ Thus, taking absolute values and applying triangle inequality gives us that $$|U_{N}| \leq |U_{N_0}| + \sum_{j=N_0}^{N-1}\left|c_j \right| \leq |U_{N_0}| + |c_{N-1}|\sum_{j=N_0}^{N-1} q^{N-j-1} \leq |U_{N_0}| + \frac{1}{1-q}|c_{N-1}|$$ Since $\frac{|c_N|}{|c_{N-1}|} \to \infty$, and $c_N \neq0$ for sufficiently large $N$, we have $|c_N| \to \infty$. Thus, dividing by $|c_N|$ gives $$0\leq \frac{|U_N|}{|c_N|}\leq \frac{|U_{N_0}|}{|c_N|} + \frac{1}{1-q} \frac{|c_{N-1}|}{|c_N|} \to 0$$ Then, $$\frac{|U_N-U_{N+1}|}{|U_N|} = \frac{|c_N|}{|U_N|} \to \infty$$ Thus, there exists an $\varepsilon>0$ and $N_1$ such that $$|U_N - U_{N+1}| \geq \varepsilon |U_N| \qquad (N \geq N_1)$$ which proves shift-nonresonance. 
\end{proof}

This lemma is useful because it turns shift-nonresonance from an additional selector condition into a direct growth check on the tail difference. 

The next result is the analytic tool used to verify Bromwich-admissibility. The finite-part construction requires not only a formal Bromwich transform, but also a reconstruction formula $f(N)=\int_{-\infty}^{\infty}\mcB[f](u)e^{Nu}\,du.$ The following standard vertical-line inversion theorem gives this reconstruction
under elementary $L^1$ hypotheses. The proof is standard, so we will not show it here. 

\begin{thm} \label{Bromwich Inversion Theorem}
Let $f$ be holomorphic in a strip containing the vertical line $\Re z=\sigma$, and set
$g_\sigma(t):=f(\sigma+it)$. Assume that $g_\sigma\in L^1(\RR)$ and that its Fourier transform
$$\widehat g_\sigma(u):=\int_{-\infty}^{\infty} f(\sigma+it)e^{-iut}\,dt$$
also belongs to $L^1(\RR)$. Define $\mcB_\sigma[f](u):=\frac{e^{-u\sigma}}{2\pi}\widehat g_\sigma(u)$. Equivalently, $$\mcB_\sigma[f](u) =\frac{1}{2\pi i}\int_{\sigma-i\infty}^{\sigma+i\infty} f(z)e^{-uz}\,dz$$ Then $$f(\sigma+it)=\int_{-\infty}^{\infty}\mcB_\sigma[f](u)e^{u(\sigma+it)}\,du$$ for every $t\in\RR$ at which $g_\sigma$ is continuous. In particular, since $f$ is holomorphic in a neighborhood of the line, this holds for every $t\in\RR$.
\end{thm}

We now apply the inversion criterion to alternating series. Let $f$ be a phase-free coefficient germ. The recurrence $I_N+I_{N+1}=f(N)$
is the recurrence naturally associated with the formal alternating series $\sum_{n=0}^{\infty}(-1)^nf(n).$ The alternating selector $\mcS_-[f]$ gives the distinguished solution of this recurrence. The theorem below gives checkable analytic hypotheses which ensure that this selector exists, reconstructs $f(N)$, and is shift-nonresonant.
Thus the abstract admissibility requirements from the finite-part section become concrete vertical-decay and growth assumptions on $f$.

\begin{thm}  \label{class theorem alt}
    Let $f\in \mcM^\flat$ be a phase free coefficient germ and suppose its canonical continuation $f(z)$ is holomorphic in a neighborhood of a closed right half plane $\Re(z)\ge 0$. Assume that for every $A>0$ and for $j=0,1,2$ there is a constant $C_A$ such that $$\left|\frac{\partial^j}{\partial t^j} f(\sigma + it) \right|\leq \frac{C_A}{(1+|t|)^3}$$ for all $0\le \sigma \le A$ and $t\in \RR$. Furthermore assume that $$ \frac{|f(N+1)|}{|f(N)|} \to \infty \qquad (N\to \infty) $$ Then, the series $$\sum_{n=0}^\infty (-1)^n f(n)$$ is $C$-summable. 
\end{thm}

\begin{proof}
    We verify the three parts of admissibility in turn: existence and independence of the Bromwich transform, reconstruction of $f(N)$, existence of the alternating selector, and finally shift-nonresonance. The vertical decay assumption implies that, for each fixed $\sigma\ge 0$,
    $$\mcB_\sigma[f](u) =\frac{1}{2\pi i} \int_{\sigma - i\infty}^{\sigma+ i\infty} f(z)e^{-uz} dz$$
    converges absolutely. We have that on the vertical line $z=\sigma+it$, $|e^{-uz}| = e^{-u\sigma}$, while
    $$\left| f(\sigma + it)\right|\leq \frac{C_A}{(1+|t|)^3}$$
    for any $A>\sigma$. This means the integral converges absolutely.The next step is to show that $\mcB_\sigma[f](u)$ is independent of $\sigma$. This is where holomorphy in the right half-plane and uniform vertical decay are used. Fix $0\le \sigma_1<\sigma_2\le A$, and integrate the holomorphic function $f(z)e^{-uz}$ around the rectangle with vertical sides $\sigma_1,\sigma_2$ and horizontal sides at heights $\pm T$. By Cauchy's theorem the contour integral is zero. On the horizontal sides one has
    $$ |f(\sigma\pm iT)e^{-u(\sigma\pm iT)}| \le e^{-u\sigma}\frac{C_A}{(1+|T|)^3}\qquad \sigma\in[\sigma_1,\sigma_2] $$
    so the horizontal integrals tend to $0$ as $T\to\infty$. Hence the two vertical integrals are equal. Thus $\mcB_\sigma[f](u)$ is independent of $\sigma\ge 0$. We therefore write this common value as $\mcB[f](u)$.

    We also need decay of $\mcB[f]$ in the Bromwich variable $u$. This follows by integrating  the vertical line. For $u\neq 0$, integrate by parts twice:
    $$ \mcB_\sigma[f](u) = \frac{e^{-u\sigma}}{2\pi} \int_{-\infty}^{\infty} f(\sigma+it)e^{-iut}\,dt $$
    Since the $j=0,1$ bounds imply that the boundary terms vanish, we get
    $$ \mcB_\sigma[f](u) = -\frac{e^{-u\sigma}}{2\pi u^2} \int_{-\infty}^{\infty} \frac{\partial^2}{\partial t^2}f(\sigma+it)e^{-iut}\,dt $$
    Using the $j=2$ bound, and using the trivial bound near $u=0$, it follows that for every $A>0$ there is a constant $C_A'$ such that
    $$ |\mcB_\sigma[f](u)| \le C_A'\frac{e^{-u\sigma}}{1+u^2} \qquad (0\le \sigma\le A,\ u\in\RR). $$
    In particular, for fixed $\sigma$, this gives
    $$|\widehat g_\sigma(u)|\le \frac{C_\sigma}{1+u^2},$$
    so $\widehat g_\sigma\in L^1(\RR)$. Also, the $j=0$ vertical decay bound shows that $g_\sigma(t)=f(\sigma+it)$ belongs to $L^1(\RR)$.

    Therefore, by the vertical line Bromwich inversion theorem \ref{Bromwich Inversion Theorem}, applied on the line $\Re z=N$ with $t=0$, we get
    $$f(N) = \int_{-\infty}^\infty \mcB_N[f](u)e^{Nu}du$$
    for every $N\ge 0$. Since $\mcB_N[f]=\mcB[f]$, this becomes
    $$f(N) = \int_{-\infty}^\infty \mcB[f](u)e^{Nu}du.$$ The same estimates also give the absolute convergence required in the definition of Bromwich-admissibility. For $u\le 0$, using the estimate with $\sigma=0$, we have $|\mcB[f](u)e^{Nu}| \le C(1+u^2)^{-1}e^{Nu}$, which is integrable on $(-\infty,0]$. For $u\ge 0$, choose $A>N+1$. Then $$|\mcB[f](u)e^{Nu}|\le C_A(1+u^2)^{-1}e^{-(A-N)u},$$ which is integrable on $[0,\infty)$. Hence$$\int_{-\infty}^{\infty}\mcB[f](u)e^{Nu}\,du$$ converges absolutely for every sufficiently large integer $N$,
    
    Then, $f$ is Bromwich admissible, since the $\mcB_\sigma[f](u)$ integral converges absolutely, $\mcB_\sigma[f](u)$ is independent of $\sigma>0$, and we have $$f(N) = \int_{-\infty}^\infty \mcB[f](u)e^{Nu}du.$$

    We now prove that $\mcS_-[f](N)$ exists for every $N\ge 0$. The same estimates also prove absolute convergence of the  selector. The independence of $\sigma$ lets us choose different vertical lines on the two halves of the $u$-axis. By definition,  $$ \mcS_-[f](N) = \int_{-\infty}^{\infty} \frac{\mcB[f](u)e^{Nu}}{1+e^u}\,du.$$ Since $\mcB_\sigma[f]$ is independent of $\sigma$, we may estimate $\mcB[f]$ using different choices of $\sigma$ on the two halves of the $u$-axis. Split the integral into $(-\infty,0]$ and $[0,\infty)$. For $u\le 0$, choose $\sigma=0$. Then
    $$ |\mcB[f](u)|\le \frac{C}{1+u^2} $$
    so
    $$ \left| \frac{\mcB[f](u)e^{Nu}}{1+e^u} \right|\le \frac{C e^{Nu}}{1+u^2} $$
    which is integrable on $(-\infty,0]$. For $u\ge 0$, choose $\sigma=A$ with $A>N+1$. Then
    $$ |\mcB[f](u)|\le C_A' \frac{e^{-Au}}{1+u^2}$$
    hence
    $$\left|\frac{\mcB[f](u)e^{Nu}}{1+e^u}\right| \le C_A' \frac{e^{-(A-N)u}}{1+u^2} $$
    which is integrable on $[0,\infty)$ because $A-N>1$. Thus $\mcS_-[f](N)$ exists.

    We then have that $c(x)=e^{i\pi x}f(x)\in\mcM$. Its phase split is
    $$c^{(+)}(x)=0,\qquad c^{(-)}(x)=f(x),$$
    which is allowed since the zero germ belongs to $\mcM^\flat$. Hence
    $$U_c(N)=(-1)^N\mcS_-[f](N).$$
    By exact inversion of $1+S$, we have
    $$\mcS_-[f](N) + \mcS_-[f](N+1) = f(N).$$
    So, we then get
    $$U_c(N) - U_c(N+1) = (-1)^N(\mcS_-[f](N) + \mcS_-[f](N+1)) = (-1)^Nf(N) = c(N).$$
    Finally, since $c(N)=(-1)^Nf(N)$, and we have
    $$\frac{|c(N)|}{|c(N-1)|}=\frac{|f(N)|}{|f(N-1)|}\to\infty,$$
    by \ref{shift sufficient ratio}, we have that $U_c$ is shift-nonresonant. Thus, $c(x)=e^{i\pi x}f(x)\in\mcM$ admits the canonical selector $U_c$, and $U_c$ is shift-nonresonant.

    To connect this to the associated series, take the recurrence presentation $a(n)= b(n)=1$, $\mu_n=f(n),$ and $ n_0=0.$
    Then the associated formal series is
    $$\sum_{n=0}^{\infty}(-1)^n f(n).$$
    Define $I_N:=(-1)^N U_c(N).$ Since $U_c(N)-U_c(N+1)=(-1)^Nf(N),$ we have $I_N+I_{N+1}=(-1)^N\bigl(U_c(N)-U_c(N+1)\bigr)=f(N).$ Thus $I_N$ is a solution of the recurrence $I_N+I_{N+1}=f(N)$. Also, $P(N,0)=(-1)^N,$  so the normalized tail is
    $$R_N=\frac{I_N}{P(N,0)}=U_c(N).$$
    Since $R_N-R_{N+1}=c(N)$ and $U_c$ is shift-nonresonant, the normalized tail is $\mcM$ admissible. Thus,
    $$\sum_{n=0}^\infty (-1)^n f(n)$$
    is $C$-summable.
\end{proof}

Thus the theorem gives a practical route to $C$-summability for rapidly growing alternating series. The vertical estimates produce the Bromwich transform and the selector, while the ratio condition verifies shift-nonresonance. The examples in the next section use this theorem exactly in this form.

We next compare $C$-summation with ordinary Borel-Laplace summation. This isn not a general equivalence theorem. The statement applies when the same Bromwich data which defines the recurrence selector also gives an ordinary Borel-Laplace integral.

The mechanism is simple. If $f(N)$ is reconstructed from $\Phi(u)=\mcB[f](u)$, then the Borel transform of $a_n=(-1)^nf(n)$ can be written as $\widehat A(t)=\int_{-\infty}^{\infty}\Phi(u)e^{-te^u}\,du.$ The Laplace integral then collapses to the same kernel which appears in the alternating selector: $\int_0^\infty e^{-t(1+e^u)}\,dt=\frac{1}{1+e^u}.$ Thus the two methods agree because they use the same Bromwich data and the same resolvent kernel.

\begin{thm}
Let $f\in\mcM^\flat$ satisfy the hypotheses of \ref{class theorem alt}, and write $\Phi(u):=\mcB[f](u)$. Assume moreover that there exists $\rho>0$ such that $$\int_{-\infty}^{\infty}|\Phi(u)|e^{\rho e^u}\,du<\infty.$$ Define $a_n:=(-1)^n f(n)$ and let $\widehat A(t):=\sum_{n=0}^\infty a_nt^n/n!$ be the ordinary Borel transform.
Then $\widehat A$ is represented near $t=0$ by $$\widehat A(t)=\int_{-\infty}^{\infty}\Phi(u)e^{-te^u}\,du,$$ this representation defines the Borel-Laplace integral on the positive real axis, and
$$ C\left(\sum_{n=0}^\infty a_n\right) = \int_0^\infty e^{-t}\widehat A(t)\,dt.
$$
\end{thm}

\begin{proof}
By the preceding class theorem, the alternating series
$\sum_{n=0}^\infty (-1)^nf(n)=\sum_{n=0}^\infty a_n$ is $C$-summable. The recurrence is
$I_N+I_{N+1}=f(N)$, with selected solution
$$
I_N=\mcS_-[f](N)
=
\int_{-\infty}^{\infty}\frac{\Phi(u)e^{Nu}}{1+e^u}\,du.
$$
For this recurrence, the normalized tail is $R_N=(-1)^NI_N$. If
$c(x)=e^{i\pi x}f(x)$, then $c^{(+)}=0$ and $c^{(-)}=f$, so the canonical selector is
$U_c(N)=(-1)^N\mcS_-[f](N)=R_N$. Hence $\FP(R)=0$, and therefore the $C$-sum is
$I_0$.

It remains to identify $I_0$ with the Borel-Laplace integral. For $|t|<\rho$, the exponential integrability assumption allows us to interchange the sum and integral:
$$
\widehat A(t)
=
\sum_{n=0}^\infty \frac{(-1)^nt^n}{n!}
\int_{-\infty}^{\infty}\Phi(u)e^{nu}\,du
=
\int_{-\infty}^{\infty}\Phi(u)e^{-te^u}\,du.
$$
For $t\ge 0$, this integral is still defined since $\Phi\in L^1$ and
$|e^{-te^u}|\le 1$.

Finally, the existence of $\mcS_-[f](0)$ gives
$$
\int_{-\infty}^{\infty}\frac{|\Phi(u)|}{1+e^u}\,du<\infty.
$$
Thus Fubini applies, and using
$\int_0^\infty e^{-t(1+e^u)}\,dt=(1+e^u)^{-1}$, we obtain
$$
\int_0^\infty e^{-t}\widehat A(t)\,dt
=
\int_{-\infty}^{\infty}\frac{\Phi(u)}{1+e^u}\,du
=
I_0.
$$
Since the $C$-sum is also $I_0$, the two values agree.
\end{proof}

This theorem explains the agreement with Borel summation in the ordinary Laplace-convergent case. The principal-value renormalon examples below are different: there the Borel transform has poles on the positive real axis, so the ordinary Laplace integral must be replaced by a principal-value prescription.

The remaining sections use these criteria in concrete examples. The Gevrey and ultra-rapid alternating examples use the first theorem: the Bromwich estimates give admissibility, and the rapid-growth ratio gives shift-nonresonance. The Borel comparison theorem explains why, in the ordinary Laplace-convergent case, the $C$-sum agrees with the Borel value. The finite-pole renormalon examples then show how the same selector framework handles principal-value singularities.

\section{Classical Examples} \label{classical examples}

We now turn from the admissibility criteria to concrete examples. The first family is the alternating factorial series $\sum_{n=0}^{\infty}(-1)^n(kn)!a^n.$ This is a useful test case because the order of growth changes with $k$. For $k=1$, the coefficients have ordinary factorial growth, while larger values of $k$ give higher Gevrey-type growth. In many classical approaches, this change in growth requires changing the transform or the summation kernel. In the recurrence-based construction, the same first-order recurrence and the same alternating selector apply for every $k$.

The example also illustrates the role of the finite part in a familiar setting. The recurrence produces a natural tail $I_N$. After normalization, the canonical selector coincides with this normalized tail, so the finite part is zero and the assigned value is the initial selected tail $I_0$.

The proof is direct. We first construct a solution of the recurrence by an integral formula, then identify the corresponding Bromwich density. The alternating selector recovers the same integral tail, so the finite-part extraction leaves no remaining constant correction.

\begin{thm}
Fix an integer $k\ge 1$, and let $a>0$. Then, when interpreted through $C$-summation,
$$
\sum_{n=0}^\infty (-1)^n (kn)! a^n
=
\int_0^\infty \frac{e^{-x}}{1+ax^k}\,dx
$$
\end{thm}

\begin{proof}
Let
$$
I_N=\int_0^\infty \frac{(ax^{k})^Ne^{-x}}{1+ax^k}\,dx.
$$
Then $I_N+I_{N+1}=(kN)!a^N$ for every $N\ge 0$, so $I_N=(-1)^N R_N$ and
$$
R_N-R_{N+1}=(-1)^N(I_N+I_{N+1})=(-1)^N (kN)!a^N.
$$
Define $c(x):=e^{i\pi x}a^x\Gamma(kx+1)\in\mcM$. Then
$c(N)=(-1)^N (kN)! a^N=R_N-R_{N+1}$ for every integer $N\ge 0$, and the phase
split of $c$ is $c^{(+)}(x)=0$ and $c^{(-)}(x)=a^x\Gamma(kx+1)$.

The next step is to show that this tail is exactly the one selected by the Bromwich construction. For this, we compute the Bromwich representation of the phase-free part $a^x\Gamma(kx+1)$.
For $\Re z>-1/k$,
$a^z\Gamma(kz+1)=\int_0^\infty e^{-x}(ax^{k})^z\,dx$. With the change of
variables $e^u=ax^k$, so that $x=a^{-1/k}e^{u/k}$ and
$dx=\frac{1}{k}a^{-1/k}e^{u/k}\,du$, this becomes
$$
a^z\Gamma(kz+1)
=
\int_{-\infty}^{\infty}
\frac{1}{k}a^{-1/k} e^{-a^{-1/k}e^{u/k}}e^{u/k}e^{zu}\,du.
$$
Comparing this with the Bromwich representation gives
$$
\mcB[c^{(-)}](u)=\frac{1}{k}a^{-1/k} e^{-a^{-1/k}e^{u/k}}e^{u/k}.
$$
Therefore
$\mcS_-[c^{(-)}](N)=\int_{-\infty}^{\infty}
\frac{\mcB[c^{(-)}](u)e^{Nu}}{1+e^u}\,du$. Changing variables back by
$e^u=ax^k$ gives
$$
\mcS_-[c^{(-)}](N)
=
\int_0^\infty \frac{e^{-x}(ax^{k})^N}{1+ax^k}\,dx
=
I_N.
$$

Since $c^{(+)}=0$, the canonical selector attached to $c$ is
$U_c(N)=(-1)^N\mcS_-[c^{(-)}](N)=(-1)^N I_N=R_N$. Thus
$R_N-U_c(N)=0$ for every $N$, so $\FP(R)=0$. It remains only to check that the selected tail satisfies the shift-nonresonance condition required for admissibility. In this example this is immediate from positivity of the integral tail. Since $a>0$, we have $I_N>0$ for every $N$, and hence
$$
|R_N-R_{N+1}|=I_N+I_{N+1}\ge I_N=|R_N|
$$
for every $N\ge 0$. Hence $U_c=R$ is shift-nonresonant with $\varepsilon=1$.
Therefore $R$ is $\mcM$-admissible, and since $\FP(R)=0$, the $C$-sum assigned
by the recurrence is $I_0$, namely
$$
\int_0^\infty \frac{e^{-x}}{1+ax^k}\,dx,
$$
as claimed.
\end{proof}

This example shows that the recurrence-based construction treats the whole family $(kn)!a^n$ uniformly. Increasing $k$ changes the growth scale of the coefficients and changes the Bromwich density, but it does not change the underlying selector mechanism. The same recurrence and the same kernel $(1+e^u)^{-1}$ produce the value for every $k$.

\section{Fast Growing Series} \label{fast grow}

The preceding examples show that the recurrence method handles factorial and higher Gevrey-type growth in a uniform way. We now consider a different kind of growth. The coefficients $e^{n^2/2}$
grow faster than $(kn)!$ for every fixed integer $k$, so this series is not naturally part of the factorial hierarchy treated above.

This example is important because it separates two ideas which are often tied together in classical summation methods: the size of the coefficients on the real axis and the regularity of the transformed data. Although $e^{n^2/2}$ grows extremely quickly, the analytic germ $e^{x^2/2}$ has a simple Bromwich representation with Gaussian density. Thus the tail difference has very large integer values, but its Bromwich data is smooth and rapidly decaying.

The example therefore shows that $C$-summability is not governed only by the real-axis growth rate of the coefficients. What matters is whether the tail-difference germ admits a canonical selector. 

The following theorem gives the ultra-rapid example. The recurrence
tail is a lognormal moment integral, and the alternating selector recovers this tail exactly.

\begin{thm}
When interpreted through $C$-summation,
$$
\sum_{n=0}^\infty (-1)^n e^{n^2/2}=\frac{1}{2}
$$
\end{thm}

\begin{proof}
Set
$$
I_n:=\int_0^\infty \frac{x^n}{1+x}\frac{1}{x\sqrt{2\pi}}e^{-(\log x)^2/2}\,dx
$$
Then $ I_N+I_{N+1} = e^{N^2/2}$, which means that $R_N:=(-1)^N I_N$. For every $N\ge 0$, 
$$
R_N-R_{N+1}
=
(-1)^N(I_N+I_{N+1})
=
(-1)^N e^{N^2/2}
$$

Next, define $c(x):=e^{i\pi x}e^{x^2/2}\in\mcM$. Then $$c(N)=(-1)^N e^{N^2/2}=R_N-R_{N+1}$$ for every integer $N\ge 0$. The phase split of $c$ is $ c^{(+)}(x)=0$ and $c^{(-)}(x)=e^{x^2/2}$. 
Now, we have that $$ e^{z^2/2} = \frac{1}{\sqrt{2\pi}}\int_{-\infty}^{\infty}e^{-u^2/2}e^{zu}\,du$$ for every $z\in\CC$. Hence, by the standard Bromwich inversion theorem,
$c^{(-)}$ is Bromwich-admissible and
$$\mcB[c^{(-)}](u)=\frac{1}{\sqrt{2\pi}}e^{-u^2/2}$$
Therefore
$$\mcS_-[c^{(-)}](N) = \int_{-\infty}^{\infty} \frac{\mcB[c^{(-)}](u)e^{Nu}}{1+e^u}\,du = \frac{1}{\sqrt{2\pi}} \int_{-\infty}^{\infty}
\frac{e^{-u^2/2}e^{Nu}}{1+e^u}\,du$$
With the change of variables $x=e^u$, this becomes
$$
\mcS_-[c^{(-)}](N)
=
\int_0^\infty \frac{x^N}{1+x}\frac{1}{x\sqrt{2\pi}}e^{-(\log x)^2/2}\,dx
=
I_N
$$

Since $c^{(+)}=0$, the canonical selector attached to $c$ is
$$
U_c(N)=(-1)^N\mcS_-[c^{(-)}](N)=(-1)^N I_N=R_N
$$
Thus
$$
R_N-U_c(N)=0
$$
for every $N$, so
$$
\FP(R)=0
$$
It remains only to check the admissibility condition not supplied by the Bromwich calculation, namely shift-nonresonance. In this example the selector equals the normalized tail, and positivity of the integral tail gives the condition directly. Since $I_N>0$ for every $N$,
$$
|R_N-R_{N+1}|=I_N+I_{N+1}\ge I_N=|R_N|
$$
for every $N\ge 0$. Since $U_c = R_N$ on integers, we have that we $U_c$ is shift-nonresonant with $\varepsilon=1$.

Therefore $R_N$ is $\mcM$-admissible, and since $\FP(R)=0$, the $C$-sum is
$$\sum_{n=0}^\infty (-1)^n e^{n^2/2} = \int_0^\infty \frac{1}{1+x}\cdot \frac{1}{\sqrt{2\pi}}\cdot \frac{1}{x} \cdot e^{-\log(x)^2/2} dx = \frac{1}{2}$$

as claimed.
\end{proof}

This example is qualitatively different from the factorial family. The
coefficient sequence $e^{n^2/2}$ grows faster than $(kn)!$ for every fixed $k$, but the method still gives a direct value because the Bromwich side is well behaved. This illustrates a central feature of $C$-summation: admissibility is not governed only by the size of the coefficients on the real axis. Very fast growth can still be summable when the associated Bromwich data has the right structure.

The same Gaussian structure also suggests a useful limiting idea. Multiplication by $e^{a x^2}$ accelerates growth on the real axis, but on the Bromwich side it acts in the opposite way: it convolves the Bromwich data with a Gaussian kernel. Thus Gaussian acceleration can turn singular or atomic Bromwich data into an ordinary density. After computing the $C$-sum of the accelerated series, one may then let $a\to0^+$, producing an Abel-type deceleration limit. The next example should therefore not be read as an ordinary $C$-sum of the
original series. Rather, it shows how the $C$-summable Gaussian-accelerated series can approach a limiting value for a series whose unaccelerated Bromwich data is not an ordinary function.

Consider $S=\sum_{n=0}^{\infty}(-1)^n e^{e^n}.$ The germ $e^{e^x}$ has the formal expansion $e^{e^x}=\sum_{k=0}^{\infty}\frac{e^{kx}}{k!}.$ Thus its Bromwich data would be a sum of point masses rather than an ordinary density. This places $S$ outside the ordinary Bromwich framework used in this paper. Gaussian acceleration replaces each point mass by a Gaussian packet, so the accelerated series becomes ordinary $C$-summable for every $a>0$.
\begin{ex}
    Consider the series $$S=\sum_{n=0}^\infty (-1)^ne^{e^n}$$ This series is not $C$-summable, since the Bromwich transform of $e^{i\pi x}e^{e^x}$ does not exist. However, we can now accelerate the series. Now fix $a>0$. Define $$S(a) = \sum_{n=0}^\infty (-1)^ne^{an^2 + e^n}$$. We now need to find the $C$-sum of this series. We know that $e^{e^x} = \sum_{k=0}^\infty \frac{e^{kx}}{k!}$. Thus, $e^{ax^2+e^x} = \sum_{k=0}^{\infty} \frac{e^{ax^2 +kx}}{k!}$. Also note that $$e^{ax^2+kx} = \int_{-\infty}^\infty \frac{1}{\sqrt{4\pi a}} \exp {\left( -\frac{(u-k)^2}{4a} \right)} e^{xu} du$$ Then,  $e^{ax^2+e^x} = \int_{-\infty}^\infty \left(\sum_{k=0}^\infty \frac{1}{k!} \frac{1}{\sqrt{4\pi a}} \exp{\left( -\frac{(u-k)^2}{4a} \right)} \right)e^{ux} du$. This means that $$\mcB[c^{(-)}](u) = \sum_{k=0}^\infty \frac{1}{k!} \frac{1}{\sqrt{4\pi a}} \exp{\left( -\frac{(u-k)^2}{4a} \right)} $$ Thus, the $C$-sum of $S(a)$ is $$\int_{-\infty}^\infty \sum_{k=0}^\infty \frac{1}{k!} \frac{1}{\sqrt{4\pi a}} \exp{\left( -\frac{(u-k)^2}{4a} \right)} \frac{1}{1+e^u} du $$ Now we will decelerate the series. We will do this by taking $a\to 0^+$ $$\lim_{a\to 0^+} C\bigg(\sum_{n=0}^\infty (-1)^ne^{an^2 + e^n} \bigg) = \lim_{a\to  0^+}\int_{-\infty}^\infty \sum_{k=0}^\infty \frac{1}{k!} \frac{1}{\sqrt{4\pi a}} \exp{\left( -\frac{(u-k)^2}{4a} \right)} \frac{1}{1+e^u} du  $$ Note that as $a \to 0^+$, we have $\frac{1}{\sqrt{4\pi a}} \exp{\left( -\frac{(u-k)^2}{4a} \right)} \to \delta(u-k)$. This means the decelerated value becomes $$\lim_{a\to 0^+} C\bigg(\sum_{n=0}^\infty (-1)^ne^{an^2 + e^n} \bigg) = \sum_{k=0}^\infty \frac{1}{k!}\int_{-\infty}^\infty \frac{\delta(u-k)}{1+e^u}du = \sum_{k=0}^\infty \frac{1}{k!(1+e^k)} $$ Thus $S$ is not directly $C$-summable in the ordinary Bromwich framework. However, the accelerated series are genuinely $C$-summable for every $a>0$, and the limit as $a\to0^+$ gives a natural Abel-type value associated to $S$: $\sum_{k=0}^{\infty}\frac{1}{k!(1+e^k)}.$ This limiting procedure is useful because it does not suppress growth in the usual Borel sense. Instead, it increases real-axis growth while smoothing the Bromwich data, and then removes the smoothing by an Abel-type limit.

\end{ex}

\section{Applications to Renormalons} \label{renormalon}

We now turn to a setting where the obstruction is not only rapid growth, but ambiguity of the summation prescription. For factorially divergent series, ordinary Borel summation first removes the $n!$-growth. However, if the Borel transform has singularities on the positive real axis, the Laplace integral runs directly into those singularities. The question is then not merely whether the integral converges, but how the singular direction should be interpreted.

Renormalon-type series give a simple model of this issue. Their coefficients have factorial growth, and their Borel transforms have poles on the natural summation ray. The purpose of this section is to show that the recurrence finite-part construction selects the Cauchy principal value in this situation. The principal value is not introduced as an external contour prescription. Instead, it appears from the non-oscillatory selector $\mcS_+[a](N)$. The pole of the kernel at $u=0$ becomes, after a change of variables, the pole
of the renormalon-type integral on the positive real axis.

The one-pole model already contains the essential mechanism. The coefficient germ $a^x\Gamma(x+1)$ has an ordinary Bromwich density, while the selector kernel $(1-e^u)^{-1}$ contributes a principal-value singularity. Under the change of variables $e^u=ax$, this singularity is carried to the pole $x=1/a$. Thus the principal value in the final integral is the image of the principal value already present in the recurrence selector.

\begin{thm}
Fix $a>0$. Then, when interpreted through $C$-summation,
$$
\sum_{n=0}^\infty n!a^n
=
\frac{e^{-1/a}}{a}Ei\!\left(\frac{1}{a}\right)
$$
where the right-hand side is understood in the principal-value sense.
\end{thm}
\begin{proof}
Set $$I_N:=PV\int_0^\infty \frac{(ax)^N e^{-x}}{1-ax}\,dx$$
Then $I_N-I_{N+1}=a^N N!$. We therefore have $R_N=I_N$, and $R_N-R_{N+1}=a^N N!$. Define
$$
c(x):=a^x\Gamma(x+1)\in\mcM.
$$
Then $c(N)=a^N N!=R_N-R_{N+1}$ for every integer $N\ge 0$. The phase split of $c$ is 
$$c^{(+)}(x)=a^x\Gamma(x+1),\qquad c^{(-)}(x)=0.$$

For $\Re z>-1$, we have $$a^z\Gamma(z+1)=\int_0^\infty e^{-x}(ax)^z\,dx.$$ With the change of variables $x=e^u/a$, $dx=\frac1a e^u\,du$, this becomes $$a^z\Gamma(z+1)=\int_{-\infty}^{\infty}\frac1a e^{-e^u/a}e^u e^{zu}\,du.$$

Hence, by the standard Bromwich inversion theorem, $c^{(+)}$ is
Bromwich-admissible and 
$$\mcB[c^{(+)}](u)=\frac1a e^{-e^u/a}e^u.$$ Therefore
$$\mcS_+[c^{(+)}](N)=PV\int_{-\infty}^{\infty} \frac{\mcB[c^{(+)}](u)e^{Nu}}{1-e^u}\,du= PV\int_{-\infty}^{\infty} \frac{\frac1a e^{-e^u/a}e^u e^{Nu}}{1-e^u}\,du.$$

Changing variables back by $x=e^u/a$, so that $e^u=ax$, gives $$\mcS_+[c^{(+)}](N)= PV\int_0^\infty \frac{e^{-x}(ax)^N}{1-ax}\,dx=I_N.$$

Since $c^{(-)}=0$, the canonical selector attached to $c$ is $$U_c(N)=\mcS_+[c^{(+)}](N)=I_N=R_N.$$
Thus $R_N-U_c(N)=0$ for every $N$, so $\FP(R)=0$. It remains to verify shift-nonresonance. We know that $$\frac{|c_{N}|}{|c_{N-1}|} = \frac{a^NN!}{a^{N-1}(N-1)!} = aN \to \infty$$

This shows that $U_c$ is shift-nonresonant by \ref{shift sufficient ratio}.

Therefore $R$ is $\mcM$-admissible, and since $\FP(R)=0$, the $C$-sum assigned
by the recurrence is $I_0$. Finally,
$$
I_0
=
PV\int_0^\infty \frac{e^{-x}}{1-ax}\,dx
=
\frac{e^{-1/a}}{a}Ei\!\left(\frac1a\right)
$$
with principal-value interpretation. Hence
$$
\sum_{n=0}^\infty n!a^n
=
\frac{e^{-1/a}}{a}Ei\!\left(\frac1a\right)
$$
in the sense of $C$-summation, as claimed.
\end{proof}

\begin{rem}
This example also exhibits the Stokes ambiguity associated with the selector. The Bromwich density is $\Phi(u)=\frac{1}{a}e^u e^{-e^u/a},$
and the non-oscillatory selector has kernel $(1-e^u)^{-1}$, which has a simple pole at $u=0$. Taking the contour above or below this pole gives two lateral selectors, whose difference is the residue contribution
$$
\mcS_+^{\uparrow}[c](N)-\mcS_+^{\downarrow}[c](N)
=
-2\pi i\operatorname*{Res}_{u=0}
\frac{\Phi(u)e^{Nu}}{1-e^u}
=
\frac{2\pi i}{a}e^{-1/a},
$$
up to the orientation convention. Since this difference is independent of $N$, it lies in the homogeneous solution space of the normalized first-order recurrence $R_N-R_{N+1}=c_N$. Thus the lateral Stokes ambiguity is precisely the homogeneous ambiguity isolated by the finite-part construction, while the principal-value selector is the symmetric choice between the two lateral selectors.
\end{rem}

The value obtained here agrees with the principal-value Borel prescription for this one-pole model. Thus the novelty is not the numerical value by itself. The point is that the principal value is selected by the same recurrence finite-part mechanism used throughout the paper. In ordinary Borel summation, once a pole lies on the positive real axis, one must add a prescription, such as a lateral sum or a principal value. Here the principal value arises from the canonical non-oscillatory selector attached to the tail difference.

This should be understood as a real symmetric value, not as a replacement for the full resurgent or lateral-summation data. Taking a principal value removes the imaginary ambiguity associated with going above or below the pole. The recurrence method is therefore not claiming to recover the Stokes information lost by this choice. Rather, it identifies the principal-value prescription as the value
naturally selected by the recurrence presentation and the finite-part
normalization.

The same mechanism extends to finite sums of simple renormalon poles. In this case the Borel-side expression has finitely many poles on the positive real axis, located at $t=1/r_j$. Throughout the next theorem, $PV\int_0^\infty$ means the Cauchy principal value taken at each of these real simple poles.

The finite-pole case also shows that the preceding calculation is not tied to a special one-pole identity. The selector is linear, so each pole contributes its own principal-value term, and the finite sum is handled by the same non-oscillatory kernel $(1-e^u)^{-1}$.

\begin{thm}
Let $r_1,\dots,r_J>0$ be distinct and let $d_1,\dots,d_J\in\CC$ be nonzero. Define $$a_n:=n!\sum_{j=1}^J d_j r_j^n.$$
Then the series $\sum_{n=0}^\infty a_n$ is $C$-summable. More precisely, its $C$-sum is $$PV\int_0^\infty e^{-t}\sum_{j=1}^J\frac{d_j}{1-r_jt}\,dt.$$
\end{thm}

\begin{proof}
Set
$$f(x):=\Gamma(x+1)\sum_{j=1}^J d_j r_j^x.$$ Then $a_n=f(n)$, and $f\in\mcM^\flat$ since each $r_j^x\Gamma(x+1)$ is a
gamma-type growth monomial. We use the recurrence presentation
$$I_N-I_{N+1}=f(N).$$ For each $j$, define $f_j(z):=d_jr_j^z\Gamma(z+1)$. Using $\Gamma(z+1)=\int_0^\infty e^{-t}t^z\,dt$ and changing variables $e^u=r_jt$, we get $$f_j(z)=
d_j\int_{-\infty}^{\infty}\frac1{r_j}e^u e^{-e^u/r_j}e^{zu}\,du.$$
Thus, $$\mcB[f](u)=\sum_{j=1}^J\frac{d_j}{r_j}e^u e^{-e^u/r_j}.$$  We also verify the selector-existence condition. For fixed $N$, the integrand $\frac{\mcB[f](u)e^{Nu}}{1-e^u}$ is smooth away from $u=0$. Near $u=0$, the numerator $\mcB[f](u)e^{Nu}$ is smooth, while $1-e^u=-u+O(u^2)$. Thus the only singularity is a simple $1/u$-type singularity, so the symmetric Cauchy principal value at $u=0$ exists. As $u\to+\infty$, each term is bounded, after division by $1-e^u$, by a constant multiple of $e^{Nu}e^{-e^u/r_j}$, which decays super-exponentially. As $u\to-\infty$, each term is $O(e^{(N+1)u})$, since $1-e^u\to1$. Therefore the selector integral is absolutely convergent away from $u=0$, and $\mcS_+[f](N)$ exists for all sufficiently large $N$. Therefore the non-oscillatory selector is defined, and is
$$\mcS_+[f](N) = PV\int_{-\infty}^{\infty} \frac{\mcB[f](u)e^{Nu}}{1-e^u}\,du. $$ Changing variables back by $e^u=r_jt$ in each term gives
$$ \mcS_+[f](N) = PV\int_0^\infty e^{-t}\sum_{j=1}^J d_j\frac{(r_jt)^N}{1-r_jt}\,dt.$$ By exact inversion of $1-S$, we have
$$\mcS_+[f](N)-\mcS_+[f](N+1)=f(N).$$ Thus $I_N:=\mcS_+[f](N)$ is a solution of the recurrence. For this recurrence, $P(N,0)=1$, so the normalized tail is $R_N=I_N$. Let $c(x):=f(x)$. Then $c^{(+)}=f$ and $c^{(-)}=0$, so the canonical selector attached to $c$ is $$U_c(N)=\mcS_+[f](N)=I_N=R_N.$$ Therefore $R_N-U_c(N)=0$ for every $N$, and hence $\FP(R)=0$. It remains only to check shift-nonresonance. Let $r_*$ be the largest of the $r_j$, with coefficient $d_*$. Since the $r_j$ are distinct, $$f(N)=\Gamma(N+1)\left(d_*r_*^N+o(r_*^N)\right),$$ so $$ \frac{|f(N)|}{|f(N-1)|}\sim Nr_*\to\infty. $$ Since $U_c(N)-U_c(N+1)=c(N)=f(N)$, the shift-ratio lemma implies that $U_c$ is
shift-nonresonant. Thus $R$ is $\mcM$-admissible, so the series is
$C$-summable.

Finally, since $\FP(R)=0$, the assigned value is $I_0$. From the formula for $\mcS_+[f](N)$ at $N=0$, we get $$ I_0= PV\int_0^\infty e^{-t}\sum_{j=1}^J\frac{d_j}{1-r_jt}\,dt.$$ This proves the claim.
\end{proof}

This theorem completes the renormalon-type comparison considered here. In the ordinary Borel-compatible case, the $C$-sum agrees with the convergent Borel-Laplace integral. In the present case, where poles lie on the positive summation ray, the same recurrence-selector framework gives the principal-value analogue.

The significance is therefore not that principal-value Borel summation was unavailable. It is that the recurrence construction reaches the same prescription from a different starting point: the normalized tail, its tail difference, and the canonical selector. This places ordinary Borel agreement, principal-value renormalon models, and the ultra-rapid examples above under the same finite-part framework.

\section{Conclusion}

The significance of the present work lies in showing that recurrence structure can serve as a basis for summation. The central point is that the ambiguity in the choice of recurrence solution can be transferred to the normalized tail and then removed by finite-part extraction. From this viewpoint, what matters is not only the growth rate of the coefficients, but the asymptotic form of the tail after normalization. This leads to a discrete framework that can accommodate several different kinds of divergence within a single definition, including classical examples, Gevrey-type series, ultra-rapid growth beyond the usual Gevrey hierarchy, and a renormalon-type model in which principal value regularization arises naturally from the construction. In this sense, the paper suggests that recurrence-based methods may complement transform-based summation theories by providing a different mechanism for assigning values to divergent expansions.

Several questions remain open. It would be valuable to obtain more systematic conditions guaranteeing admissibility and to better understand the shift behavior that governs stability under finite truncation. It would also be natural to strengthen the comparison with classical summation theories, both by enlarging the class in which agreement with Borel summation can be proved and by clarifying the relation with non-Borel-summable settings. Another natural direction is to seek analogous constructions for higher-order recurrences and more general discrete systems. The renormalon example also suggests possible applications to perturbative quantum field theory, where one seeks canonical prescriptions for asymptotic expansions with singularities on the positive Borel axis. Together, these questions suggest several directions in which the recurrence viewpoint developed here could be extended.

\section{Acknowledgments}
The author thanks Professor Ovidiu Costin for his guidance on summation and Professor Brian Conrad for his advice on paper presentation.

\bibliographystyle{alpha}
\bibliography{2025contfracmethod}
\cite{euler2012introduction}
\cite{compton1973integral}
\cite{whittaker2020course}
\cite{pilehrood2013continuedfractionexpansioneulers}
\cite{cuyt2008handbook}
\cite{c39be6f6-2bdf-30c1-bebf-7068a12748e0}
\cite{cvijovic1997continued}
\cite{costin2008asymptotics}
\cite{ecalle1981fonctions}
\cite{beneke1999renormalons}
\cite{hardy2024divergent}
\cite{baker1961pade}
\cite{wall2018analytic}
\cite{balser2006divergent}
\cite{t1979can}
\end{document}